\RequirePackage{fix-cm}
\documentclass[smallextended]{svjour3}       
\smartqed  
\usepackage{amssymb}
\usepackage{graphicx}
\usepackage{enumerate}
\begin{document}

\title{A Series of High Order Quasi-Compact Schemes for Space
Fractional Diffusion Equations Based on the Superconvergent Approximations for Fractional Derivatives
\thanks{This work was supported by the National Natural Science
Foundation of China under Grant No. 11271173.}
}

\titlerunning{High Order Quasi-Compact Schemes}        

\author{Lijing Zhao         \and
        Weihua Deng 
}


\institute{Lijing Zhao \at
              School of Mathematics and Statistics, Gansu Key Laboratory of Applied Mathematics and Complex Systems,
              Lanzhou University, Lanzhou 730000, P.R. China \\
              \email{zhaojane8836@gmai.com}           
           \and
           Weihua Deng \at
              School of Mathematics and Statistics, Gansu Key Laboratory of Applied Mathematics and Complex Systems,
              Lanzhou University, Lanzhou 730000, P.R. China \\
               \email{dengwh@lzu.edu.cn, dengwhmath@aliyun.com}
}

\date{Received: date / Accepted: date}

\maketitle

\begin{abstract}

Based on the superconvergent approximation at some point (depending on the fractional order $\alpha$, but not belonging to the mesh points) for Gr\"{u}nwald discretization  to fractional derivative, we develop a series of high order quasi-compact schemes for space fractional diffusion equations. Because of the quasi-compactness of the derived schemes, no points beyond the domain are used for all the high order schemes including second order, third order, fourth order, and even higher order schemes; moreover, the algebraic equations for all the high order schemes have the completely same matrix structure. The stability and convergence analysis for some typical schemes are made; the techniques of treating the nonhomogeneous boundary conditions are introduced; and extensive numerical experiments are performed to confirm the theoretical analysis or verify the convergence orders.

\end{abstract}

\section{Introduction}
\label{sec:1}

Fractional derivatives have been widely applied to model the problems in physics \cite{Barkai:00,Blumen:89,Bouchaud:90,Chaves:98,Klafter:87,Meerschaert:02,Saichev:97},
finance \cite{Raberto:02,Sabatelli:02}, and hydrology \cite{Benson:00,Schumer:03}, especially to
the anomalous diffusion or dispersion, where a particle plume spreads at a rate inconsistent with
the classical Brownian motion \cite{Chechkin:02,Krepysheva:06,Negrete:03}.
The suitable mathematical models are the generalization to the classical diffusion equations formally replacing the classical first order derivative in time by the Caputo fractional derivative of order $\alpha\in(0,1)$, and the second order derivative in space by the Riemann-Liouville fractional derivative of order $\alpha\in(1,2]$. While, unlike the classical counterparts,
because of the nonlocal properties of fractional
operators, obtaining the analytical solutions of the fractional partial differential equations (PDEs) is more challenging or sometimes even
impossible; or the obtained analytical solutions are just expressed by transcendental functions
or infinite series. So, efficiently solving the fractional PDEs naturally becomes an urgent topic.

Over the last decades, the finite difference methods have achieved some developments in solving the
fractional differential equations, e.g., \cite{Celik:12,Deng:07,Meerschaert:04,Yang:10,Zhao:13}. The Riemann-Liouville space fractional derivative
can be naturally discretized by the standard Gr\"{u}nwald-Letnikov formula \cite{Podlubny:99} with first order accuracy, but the finite difference schemes derived by the discretization are unconditionally unstable for the initial value
problems including the implicit schemes that are well known
to be stable most of the time for classical derivatives \cite{Meerschaert:04}. To remedy the instability, Meerschaert and Tadjeran in \cite{Meerschaert:04} firstly propose the shifted Gr\"{u}nwald-Letnikov formula to approximate fractional advection-dispersion flow equations with still first order of accuracy.
Recently, the high order approximations to space fractional derivatives are studied. Using the idea of second order central difference formula, a second order discretization for Riemann-Liouville fractional derivative is established in \cite{Sousa:11}, and the scheme is extended to two dimensional two-sided space fractional convection diffusion equation in finite domain in \cite{Chen:13}.  By assembling the
Gr\"{u}nwald difference operators with different weights and shifts,  a class of stable second
order discretizations for Riemann-Liouville space fractional derivative is developed in \cite{Tian:12,Li:13}, and an abstract way of discussing the stability and convergence of the discretizations can be seen in \cite{Baeumer:13}; its corresponding third order quasi-compact scheme is given in \cite{Zhou:13}. Allowing to use the points outside of the domain,  a class of second, third and fourth order approximations for Riemann-Liouville space fractional derivatives are derived in \cite{Chen2:13} by using the weighted and shifted Lubich difference operators.

Comparing with the classical PDEs, the high order finite difference schemes get more striking benefits than the low order ones in solving the fractional PDEs. The reason is that the high order schemes can keep the same computational cost as the first order ones but greatly improve the accuracy. Usually, the high and low order schemes have the same algebraic structures, e.g., $(\textbf{T}-\textbf{A})\textbf{U}^{n+1}=(\textbf{T}+\textbf{A})\textbf{U}^{n}+\textbf{b}^{n+1}$, where $\textbf{T}$ is tri-diagonal, and $\textbf{A}$ is Toeplitz-like full matrix. Even though the matrix $\textbf{A}$ is full because of the nonlocal property of fractional operator, the so-called compactness to the schemes for classical differential equation still makes sense here since it can make the high order schemes avoid using the points outside of the domain and the corresponding schemes reduce to the classical compact schemes when the order of fractional derivative is taken as an integer. And we call this kind of schemes high order quasi-compact schemes. The superconvergence of the Gr\"{u}nwald discretizaton to the Riemann-Liouville derivative at some special point is introduced in Sec. 8.2 of  \cite{Oldham:74} and a further discussion is given in \cite{Nasir:13}. Using the superconvergence of the Gr\"{u}nwald discretization, in this paper we develop a series of quasi-compact second order, third order, and fourth order schemes for space fractional diffusion equation:
\begin{equation}\label{equation1.1}
\left\{ \begin{array}{lll}
\frac{\partial u(x,t) }{\partial t}&=&K_1\,_{x_L}D_x^{\alpha}u(x,t)+K_2\, _{x}D_{x_R}^{\alpha}u(x,t)+f(x,t) \\
&& ~~~~ ~~~~ ~~~ {\rm for}~~~ (x,t) \in (x_L,x_R)\times (0,T),\\
u(x,0) &=&u_0(x) ~~~~ {\rm for}~~~ x \in [x_L,x_R], \\
u(x_L,t)&=&\phi_{L}(t) ~~~~{\rm for}~~~ t\in[0,T],\\
u(x_R,t)&=&\phi_{R}(t) ~~~~{\rm for}~~~ t\in[0,T],
\end{array} \right.
\end{equation}
where $_{x_L}D_x^{\alpha}$ and $_{x}D_{x_R}^{\alpha}$ are, respectively, left and right Riemann-Liouville fractional
derivatives with $1<\alpha \leq 2$. The diffusion
coefficients $K_1$ and $K_2$ are nonnegative constants with $K_1^2+K_2^2\neq 0$.
The left and right Riemann-Liouville  fractional
derivatives of the function $u(x)$ on $[a,b]$, $-\infty \leq a <b\leq \infty$ are, respectively, defined by \cite{Podlubny:99,Samko:93}
\begin{equation}\label{equation1.2}
_{a}D_x^{\alpha}u(x)= D^{m}{}_{a}D_{x}^{-(m-\alpha)}u(x)
~~{\rm and}~~~ _{x}D_{b}^{\alpha}u(x)=(-D)^{m}{}_{x}D_{b}^{-(m-\alpha)}u(x),
\end{equation}
where $\alpha \in (m-1,m)$; and
\begin{equation}\label{equation1.3}
{}_{a}D_{x}^{-\gamma}u(x)
=\frac{1}{\Gamma(\gamma)}\int_{a}\nolimits^x{\left(x-\xi\right)^{\gamma-1}}{u(\xi)}d\xi,  ~~~~~~ \gamma > 0,
\end{equation}
and
\begin{equation}\label{equation1.4}
{}_{x}D_{b}^{-\gamma}u(x)
= \frac{1}{\Gamma(\gamma)}
\int_{x}\nolimits^{b}{\left(\xi-x\right)^{\gamma-1}}{u(\xi)}d\xi, ~~~~~~\gamma > 0,
\end{equation}
are the $\gamma$-th left and right Riemann-Liouville fractional integrals, respectively.

The paper is organized as follows. In Section
\ref{sec:2}, we derive a series of second and third order approximations to the linear combinations of the Riemann-Liouville space fractional derivatives and show the basic ways to derive the higher order schemes. In Section \ref{sec:3}, the derived high order schemes are applied to solve space fractional diffusion problem, and the stability and convergence analyses for some of the schemes are
performed.
Some numerical results are given in Section \ref{sec:4} to confirm the theoretical analyses and convergence orders; in particular, the equation with nonhomogeneous boundary conditions is also numerically solved by using the techniques introduced in the Appendix which help keeping the desired convergence orders. We conclude the paper with some remarks in the last section.

\section{Series of high order approximations to the linear combinations of Riemann-Liouville space fractional derivatives}
\label{sec:2}

For a real number $\alpha \in \mathcal{R}$ and $f(x)\in C[a,b]$, the fractional derivatives $\,_{a}^G D_x^{\alpha}f(x)$
and $\,_{x}^G D_{b}^{\alpha}f(x)$
are defined by the standard left and right Gr\"{u}nwald-Letnikov formulations \cite{Podlubny:99}
\begin{equation}\label{equation2.01}
\,_{a}^G D_x^{\alpha}f(x)=\lim_{h\rightarrow 0} \frac{1}{h^\alpha}\sum_{k=0}^{[\frac{x-a}{h}]}g_k^{(\alpha)}f(x-kh),
\end{equation}
\begin{equation}
\,_{x}^G D_{b}^{\alpha}f(x)=\lim_{h\rightarrow 0} \frac{1}{h^\alpha}\sum_{k=0}^{[\frac{b-x}{h}]}g_k^{(\alpha)}f(x+kh),
\end{equation}
where $g_k^{(\alpha)}=(-1)^k\left( \begin{array}{c}\alpha \\ k\end{array} \right )$ are the coefficients of the power series of the generating function $(1-\zeta)^{\alpha}$,
 and they can be calculated by the following recursively formula
\begin{equation}\label{equation2.02}
g_0^{\alpha}=1, ~~~~g_k^{(\alpha)}=\left(1-\frac{\alpha+1}{k}\right)g_{k-1}^{\alpha},~~k \geq 1.
\end{equation}
If $a=-\infty$ or $b=+\infty$, then \cite{Samko:93}
\begin{equation}\label{equation2.11}
\,_{-\infty}^G D_x^{\alpha}f(x)=\lim_{h\rightarrow 0} \frac{1}{h^\alpha}\sum_{k=0}^{\infty}g_k^{(\alpha)}f(x-kh),
\end{equation}
\begin{equation}
\,_{x}^G D_{+\infty}^{\alpha}f(x)=\lim_{h\rightarrow 0} \frac{1}{h^\alpha}\sum_{k=0}^{\infty}g_k^{(\alpha)}f(x+kh).
\end{equation}
For the issue of numerical stability, the shifted Gr\"{u}nwald difference operator
\begin{equation}\label{equation2.13}
\delta^{\alpha}_{h,p}f(x):= \frac{1}{h^\alpha}\sum_{k=0}^{\infty}g_k^{(\alpha)}f(x-(k-p)h)
\end{equation}
is introduced to approximate the left Riemann-Liouville fractional derivative
(and $\sigma^{\alpha}_{h,p}f(x):= \frac{1}{h^\alpha}\sum_{k=0}^{\infty}g_k^{(\alpha)}f(x+(k-p)h)$ to the right Riemann-Liouville fractional derivative) with first order accuracy \cite{Meerschaert:04}. And it is clear that
\begin{equation}\label{equation2.14}
\delta^{\alpha}_{h,p}f(x)=\delta^{\alpha}_{h,p+q}f(x-qh),~~~q\in \mathbf{Z}.
\end{equation}

\begin{remark}\label{remark2.01}
 For the functions defined on a bounded interval $[a, b]$, sometimes we discuss them in $(-\infty,b]$ or $[a,+\infty)$ by zero extending their definitions. In the following, we do not restate this.
\end{remark}

\subsection{Some of the second order approximations}\label{subsec:2.1}

In \cite{Nasir:13,Oldham:74}, it is found that the Gr\"{u}nwald approximation has a superconvergent point. Based on this fact, in this subsection we derive a series of second order quasi-compact approximations for the combined left Riemann-Liouville fractional derivatives. For the right Riemann-Liouville
fractional derivative, the same results can be obtained if $\delta^{\alpha}_{h,p}f(x),~p=-1,0,1$ is substituted by $\sigma^{\alpha}_{h,p}f(x),~p=-1,0,1$, and $f(x-\cdot)$ by $f(x+\cdot)$.

By realigning the equi-spaced grid points, Nasir et al in \cite{Nasir:13} refer to that $\delta^{\alpha}_{h,1}f(x)$ can
approximate $\,_{-\infty}D_{x}^{\alpha}f(x+\beta h)$ with second order accuracy, where $\beta=1-\frac{\alpha}{2}$, i.e.,
\begin{equation}\label{equation2.15}
\,_{-\infty}D_{x}^{\alpha}f(x+\beta h)=\delta^{\alpha}_{h,1}f(x)+O(h^2).
\end{equation}
Supposing $\,_{-\infty}D_{x}^{\alpha}f(x)\in C^2(\mathbb{R})$, then by the Taylor series expansions there exists
\begin{eqnarray}\label{equation2.16}
&&\,_{-\infty}D_{x}^{\alpha}f(x+\beta h)
\nonumber\\
&=&\lambda^{(\beta)}_{m,n}\,_{-\infty}D_{x}^{\alpha}f(x+(m-1)h)
+\lambda^{(\beta)}_{n,m}\,_{-\infty}D_{x}^{\alpha}f(x+(n-1)h)+O(h^2),~
\end{eqnarray}
where
\begin{equation}\label{equation2.16*1}
\lambda^{(\beta)}_{m,n}=\frac{\beta-n+1}{m-n},~~~\lambda^{(\beta)}_{n,m}=\frac{\beta-m+1}{n-m},~~~m\neq n.
\end{equation}
Combining (\ref{equation2.15}) and (\ref{equation2.16}) leads to a second order approximation
\begin{eqnarray}\label{equation2.17}
&&\lambda^{(\beta)}_{m,n}\,_{-\infty}D_{x}^{\alpha}f(x+(m-1)h)+\lambda^{(\beta)}_{n,m}\,_{-\infty}D_{x}^{\alpha}f(x+(n-1)h)
\nonumber\\
&=&\delta^{\alpha}_{h,1}f(x)+O(h^2).
\end{eqnarray}
We further derive the next second order approximation. First, we have
\begin{eqnarray}\label{equation2.18}
\,_{-\infty}D_{x}^{\alpha}f(x)&=&
\xi^{(\beta)}_{p,q}\,_{-\infty}D_{x}^{\alpha}f(x+(p-1)h+\beta h)
\nonumber\\
&&+\xi^{(\beta)}_{q,p}\,_{-\infty}D_{x}^{\alpha}f(x+(q-1)h+\beta h)+O(h^2),
\end{eqnarray}
where
\begin{equation}\label{equation2.18*1}
\xi^{(\beta)}_{p,q}=\frac{\beta+q-1}{q-p},~~~\xi^{(\beta)}_{q,p}=\frac{\beta+p-1}{p-q},~~~p\neq q.
\end{equation}
From (\ref{equation2.14}), (\ref{equation2.15}), and (\ref{equation2.18}), there exists the second order approximation
\begin{eqnarray}\label{equation2.19}
&&\,_{-\infty}D_{x}^{\alpha}f(x)
\nonumber\\
&=&\xi^{(\beta)}_{p,q}\delta^{\alpha}_{h,1}f\big(x+(p-1)h\big)+\xi^{(\beta)}_{q,p}\delta^{\alpha}_{h,1}f\big(x+(q-1)h\big)+O(h^2)
\nonumber\\
&=&\xi^{(\beta)}_{p,q}\delta^{\alpha}_{h,p}f(x)+\xi^{(\beta)}_{q,p}\delta^{\alpha}_{h,q}f(x)+O(h^2).
\end{eqnarray}
From (\ref{equation2.14}), (\ref{equation2.15}), (\ref{equation2.16}), and (\ref{equation2.18}), we get
\begin{eqnarray}\label{equation2.110}
&&\xi^{(\beta)}_{p,q}\big[\lambda^{(\beta)}_{m_1,n_1}\,_{-\infty}D_{x}^{\alpha}f\big(x+(p-1)h+(m_1-1)h\big)
\nonumber\\
&&~~~~~~+\lambda^{(\beta)}_{n_1,m_1}\,_{-\infty}D_{x}^{\alpha}f\big(x+(p-1)h+(n_1-1)h\big)\big]
\nonumber\\
&&+\xi^{(\beta)}_{q,p}\big[\lambda^{(\beta)}_{m_2,n_2}\,_{-\infty}D_{x}^{\alpha}f\big(x+(q-1)h+(m_2-1)h\big)
\nonumber\\
&&~~~~~~~+\lambda^{(\beta)}_{n_2,m_2}\,_{-\infty}D_{x}^{\alpha}f\big(x+(q-1)h+(n_2-1)h\big)\big]
\nonumber\\
&=&\xi^{(\beta)}_{p,q}\big[\lambda^{(\beta)}_{m_1,n_1}\,_{-\infty}D_{x}^{\alpha}f\big(x+(p+m_1-2)h\big)
\nonumber\\
&&~~~~~~+\lambda^{(\beta)}_{n_1,m_1}\,_{-\infty}D_{x}^{\alpha}f\big(x+(p+n_1-2)h\big)\big]
\nonumber\\
&&+\xi^{(\beta)}_{q,p}\big[\lambda^{(\beta)}_{m_2,n_2}\,_{-\infty}D_{x}^{\alpha}f\big(x+(q+m_2-2)h\big)
\nonumber\\
&&~~~~~~+\lambda^{(\beta)}_{n_2,m_2}\,_{-\infty}D_{x}^{\alpha}f\big(x+(q+n_2-2)h\big)\big]
\nonumber\\
&=&\xi^{(\beta)}_{p,q}\delta^{\alpha}_{h,1}f\big(x+(p-1)h\big)+\xi^{(\beta)}_{q,p}\delta^{\alpha}_{h,1}f\big(x+(q-1)h\big)+O(h^2)
\nonumber\\
&=&\xi^{(\beta)}_{p,q}\delta^{\alpha}_{h,p}f(x)+\xi^{(\beta)}_{q,p}\delta^{\alpha}_{h,q}f(x)+O(h^2),
\end{eqnarray}
where $p\neq q,~m_1\neq n_1,~m_2\neq n_2$.

In fact, (\ref{equation2.110}) implies (\ref{equation2.19}).
So now we have obtained two types of second order approximations for the combined Riemann-Liouville fractional derivatives:
\begin{eqnarray}\label{equation2.111}
&&\lambda^{(\beta)}_{m,n}\,_{-\infty}D_{x}^{\alpha}f(x+(m-1)h)+
\lambda^{(\beta)}_{n,m}\,_{-\infty}D_{x}^{\alpha}f(x+(n-1)h)
\nonumber\\
&=&\delta^{\alpha}_{h,1}f(x)+O(h^2),
\end{eqnarray}
where $m\neq n$; and
\begin{eqnarray}\label{equation2.114}
&&\xi^{(\beta)}_{p,q}\big[\lambda^{(\beta)}_{m_1,n_1}\,_{-\infty}D_{x}^{\alpha}f\big(x+(p+m_1-2)h\big)
\nonumber\\
&&~~~~~~+\lambda^{(\beta)}_{n_1,m_1}\,_{-\infty}D_{x}^{\alpha}f\big(x+(p+n_1-2)h\big)\big]
\nonumber\\
&&+\xi^{(\beta)}_{q,p}\big[\lambda^{(\beta)}_{m_2,n_2}\,_{-\infty}D_{x}^{\alpha}f\big(x+(q+m_2-2)h\big)
\nonumber\\
&&~~~~~~+\lambda^{(\beta)}_{n_2,m_2}\,_{-\infty}D_{x}^{\alpha}f\big(x+(q+n_2-2)h\big)\big]
\nonumber\\
&=&\xi^{(\beta)}_{p,q}\delta^{\alpha}_{h,p}f(x)+\xi^{(\beta)}_{q,p}\delta^{\alpha}_{h,q}f(x)+O(h^2),
\end{eqnarray}
where $p\neq q,~m_1\neq n_1,~m_2\neq n_2$.
For convenience, we assume that $m < n,~p < q,~m_1< n_1$, and $m_2< n_2$.

\begin{remark}\label{remark2.1}
It should be noted that the operator (\ref{equation2.19}) can also be derived by the way of weighting and shifting
Gr\"{u}nwald difference operator \cite{Tian:12}.
\end{remark}

In the above the general second order schemes are presented, in practice we are more interested in the quasi-compact (not using the points outside of the domain) ones, which have the form
\begin{eqnarray}\label{equation2.115}
&&c_{-1}\,_{-\infty}D_{x}^{\alpha}f(x-h)+c_{0}\,_{-\infty}D_{x}^{\alpha}f(x)+c_{1}\,_{-\infty}D_{x}^{\alpha}f(x+h)
\nonumber\\
&=&d_{p}\delta^{\alpha}_{h,p}f(x)+d_{q}\delta^{\alpha}_{h,q}f(x)+O(h^2),
\end{eqnarray}
 where $c_{-1}+c_{0}+c_{1}=d_{p}+d_{q},~|p|\leq 1,~ |q| \leq 1$.

The quasi-compact approximations corresponding to (\ref{equation2.17}), where $(m,n)$ is taken as $(0,1)$, $(1,2)$, and $(0,2)$, respectively, are
\begin{equation}\label{equation2.116}
\lambda^{(\beta)}_{0,1}\,_{-\infty}D_{x}^{\alpha}f(x-h)+
\lambda^{(\beta)}_{1,0}\,_{-\infty}D_{x}^{\alpha}f(x)
=\delta^{\alpha}_{h,1}f(x)+O(h^2),
\end{equation}
\begin{equation}\label{equation2.117}
\lambda^{(\beta)}_{1,2}\,_{-\infty}D_{x}^{\alpha}f(x)+
\lambda^{(\beta)}_{2,1}\,_{-\infty}D_{x}^{\alpha}f(x+h)
=\delta^{\alpha}_{h,1}f(x)+O(h^2),
\end{equation}
and
\begin{equation}\label{equation2.1177}
\lambda^{(\beta)}_{0,2}\,_{-\infty}D_{x}^{\alpha}f(x-h)+
\lambda^{(\beta)}_{2,0}\,_{-\infty}D_{x}^{\alpha}f(x+h)
=\delta^{\alpha}_{h,1}f(x)+O(h^2).
\end{equation}

The values of the parameters $(p,q)$, $(m_1,n_1)$, and $(m_2,n_2)$ to generate the quasi-compact schemes from (\ref{equation2.114}) are listed in Table \ref{tables2.1}; each group of parameters corresponds to a specific different quasi-compact approximation.

\begin{table}
\caption{The parameters $(p,q)$, $(m_1,n_1)$, and $(m_2,n_2)$ corresponding to the second order quasi-compact approximations of (\ref{equation2.114})}
\label{tables2.1}
\begin{tabular}{c|cc||c|cc||c|cc}
\hline\noalign{\smallskip}
$(p,q)$ & $(m_1,n_1)$ & $(m_2,n_2)$&$(p,q)$ & $(m_1,n_1)$ & $(m_2,n_2)$ &$(p,q)$ & $(m_1,n_1)$ & $(m_2,n_2)$ \\
\noalign{\smallskip}\hline\noalign{\smallskip}
&    (1,2)& (0,1) &    &(2,3)& (0,1) &    &(2,3)& (1,2)\\
&    (1,2)& (0,2) &    &(2,3)& (0,2) &    &(2,3)& (1,3)\\
&    (1,2)& (1,2) &    &(2,3)& (1,2) &    &(2,3)& (2,3)\\
&    (1,3)& (0,1) &    &(2,4)& (0,1) &    &(2,4)& (1,2)\\
(0,1)&(1,3)& (0,2) & (-1,1)&   (2,4)& (0,2)& (-1,0)&(2,4)& (1,3)\\
&    (1,3)& (1,2) &    &(2,4)& (1,2) &    &(2,4)& (2,3)\\
&    (2,3)& (0,1) &    &(3,4)& (0,1) &    &(3,4)& (1,2)\\
&    (2,3)& (0,2) &    &(3,4)& (0,2) &    &(3,4)& (1,3)\\
\noalign{\smallskip}\hline
\end{tabular}
\end{table}

If $\,_{-\infty}D_{x}^{\alpha}f(x)\in C^3(\mathbb{R})$,  from (\ref{equation2.15}) we can derive a new quasi-compact second order approximation. Since
\begin{eqnarray}\label{equation2.118}
&&\,_{-\infty}D_{x}^{\alpha}f(x+\beta h)
\nonumber\\
&=&\lambda_{-1,0,1}^{(\beta)}\,_{-\infty}D_{x}^{\alpha}f(x-h)
+\lambda_{0,1,-1}^{(\beta)}\,_{-\infty}D_{x}^{\alpha}f(x)
\nonumber\\
&&+\lambda_{1,-1,0}^{(\beta)}\,_{-\infty}D_{x}^{\alpha}f(x+h)+O(h^3),
\end{eqnarray}
where
\begin{equation}\label{equation2.119}
\lambda_{-1,0,1}^{(\beta)}=-\frac{\beta(1-\beta)}{2},
~~~\lambda_{0,1,-1}^{(\beta)}=1-\beta^2,
~~~\lambda_{1,-1,0}^{(\beta)}=\frac{\beta(1+\beta)}{2};
\end{equation}
then the obtained approximation is
\begin{eqnarray}\label{equation2.120}
&&\lambda_{-1,0,1}^{(\beta)}\,_{-\infty}D_{x}^{\alpha}f(x-h)+
\lambda_{0,1,-1}^{(\beta)}\,_{-\infty}D_{x}^{\alpha}f(x)+
\lambda_{1,-1,0}^{(\beta)}\,_{-\infty}D_{x}^{\alpha}f(x+h)
\nonumber\\
&=&\delta^{\alpha}_{h,1}f(x)+O(h^2).
\end{eqnarray}

\subsection{Asymptotic expansions for the truncation errors of Gr\"{u}nwald approximations}\label{subsec:2.2}

For the convenience of obtaining higher order approximations later, now we make the detailed asymptotic expansions for the truncation errors of Gr\"{u}nwald approximations.
\begin{lemma}\label{lemma2.21}
Let $m-1\leq \alpha<m$, $m,n \in \mathbb{N}^{+}$, $f(x)\in C^{n+m-1}[a,b]$, $D^{n+m}f(x) \in L^1[a,b]$, and $D^{k}f(a)=D^{k}f(b)=0,~k=0,1,\cdots,n+m-1$. Then for any integer $p$ and a real parameter $\gamma$, we have
\begin{equation}\label{equation2.21}
\,_{a}D_{x}^{\alpha}f(x+\gamma h)-\delta_{h,p}^\alpha f(x)
=\sum_{l=1}^{n-1}a_l(\gamma,p)\,_{a}D_{x}^{\alpha+l}f(x)h^l+O(h^n)
\end{equation}
for $x+\gamma h \in [a,b]$, where $a_l(\gamma,p)$ are the coefficients of the power series of the function
$W_{\alpha,p}(z)=\bigg(e^{\gamma z}-e^{pz}\big(\frac{1-e^{-z}}{z}\big)^{\alpha}\bigg)$, i.e.,
$W_{\alpha,p}(z)=\sum_{l=0}^{\infty}a_l z^l$.
\end{lemma}

\begin{remark}\label{remark2.21}
It can be noted that in \cite{Tadjeran:06}, a similar result ($\gamma=0$) is given under the condition that
``let $1<\alpha<2$, $f\in C^{n+3}(\mathbb{R})$ such that all derivatives of $f$ up to order $n+3$ belong to $L^1(\mathbb{R})$",
which can be interpreted as ``Let $f\in C^{n+3}[a,b]$ and $D^{k}f(a)=D^{k}f(b)=0,~k=0,1,\cdots,n+3$", if the function $f(x)$ is defined on a bounded interval $[a, b]$. So Lemma \ref{lemma2.21} to be proven
here holds under a weaker condition.

It can be noticed that if $g(x)\in L^1[a,b]$, then $\widehat{g}(\omega)\in L^1(\mathbb{R})$,
where $\widehat{g}(\omega)=\mathcal{F}(g)(\omega)$, i.e.,
\begin{eqnarray*}
\widehat{g}(\omega)=\int_{\mathbb{R}}e^{i\omega x }g(x)dx.
\end{eqnarray*}
And
\begin{equation}\label{equation2.21*02}
\mathcal{F}[\,_{a}D_{x}^{-\alpha}g(x)](\omega)=\mathcal{F}[\,_{-\infty}D_{x}^{-\alpha}g(x)](\omega)=
(-i\omega)^{-\alpha}\widehat{g}(\omega)\in L^1(\mathbb{R})
\end{equation}
holds \cite{Podlubny:99,Samko:93}. But the similar statement for fractional derivative
\begin{equation}\label{equation2.21*03}
\mathcal{F}[\,_{a}D_{x}^{\alpha}g(x)](\omega)=\mathcal{F}[\,_{-\infty}D_{x}^{\alpha}g(x)](\omega)=
(-i\omega)^{\alpha}\widehat{g}(\omega)
\end{equation}
is NOT true, unless we make more requirements to $g(x)$, e.g., $g(x)$ and its several derivatives vanish at the end points of the interval.
\end{remark}

\begin{proof}

Firstly, it is well known that \cite{Podlubny:99} for $0\leq l \leq n$, if $f(x)\in C^{m+l-1}[a,b]$, $D^{m+l}f(x) \in L^1[a,b]$, and $D^{k}f(a)=0,~k=0,\cdots,m+l-1$, then
\begin{equation}
\,_{a}D_x^{\alpha+l}f(x)=D^{m+l}\,_{a}D_x^{\alpha-m}f(x)=\,_{a}D_x^{\alpha-m}D^{m+l}f(x)\in L^1[a,b];
\end{equation}
and it is clear that $\widehat{f}(\omega)\in L^1(\mathbb{R})$; further requiring that $D^{k}f(b)=0,~k=0,\cdots,m+l-1$
and combining with (\ref{equation2.21*02}) result in
\begin{eqnarray}
&&\mathcal{F}[\,_{a}D_{x}^{\alpha+l}f(x)](\omega)
\nonumber\\
&=&\mathcal{F}[\,_{a}D_x^{\alpha-m}D^{m+l}f(x)](\omega)
\nonumber\\
&=&(-i\omega)^{\alpha-m}\mathcal{F}[D^{m+l}f(x)](\omega)
\nonumber\\
&=&(-i\omega)^{\alpha-m}(-i\omega)^{m+l}\hat{f}(\omega)
\nonumber\\
&=&(-i\omega)^{\alpha+l}\widehat{f}(\omega)\in L^1(\mathbb{R}).
\end{eqnarray}
Next, we show that $\mathcal{F}\big[\,_{a}D_{x}^{\alpha}f(x+\gamma h)-\delta_{h,p}^\alpha f(x)\big](\omega)\in L^1(\mathbb{R})$ which means that
 $\,_{a}D_{x}^{\alpha}f(x+\gamma h)-\delta_{h,p}^\alpha f(x)\in L^1(\mathbb{R})$; and then (\ref{equation2.21}) holds.

Since
\begin{eqnarray*}
\mathcal{F}[f(x-a)](\omega)=e^{ia\omega}\widehat{f}(\omega),
\end{eqnarray*}
and that
\begin{eqnarray*}
(1-z)^{\alpha}=\sum_{k=0}^{+\infty}g_k^{(\alpha)}z^k
\end{eqnarray*}
converges absolutely for $|z|\leq 1$, we have
\begin{eqnarray}\label{equation2.21*2}
&&\mathcal{F}\big[\,_{a}D_{x}^{\alpha}f(x+\gamma h)-\delta_{h,p}^\alpha f(x)\big](\omega)
\nonumber\\
&=&\Big((-i\omega)^{\alpha}e^{-i\omega\gamma h}-\frac{e^{-i\omega p h}}{h^{\alpha}}
\sum_{k=0}^{+\infty}g_k^{(\alpha)}e^{i\omega kh}\Big)\widehat{f}(\omega)
\nonumber\\
&=&\Big((-i\omega)^{\alpha}e^{-i\omega\gamma h}-\frac{e^{-i\omega p h}}{h^{\alpha}}
(1-e^{i\omega h})^{\alpha}\Big)\widehat{f}(\omega)
\nonumber\\
&=&\Big(e^{-i\omega\gamma h}-e^{-i\omega p h}
\big(\frac{1-e^{i\omega h}}{-i\omega h}\big)^{\alpha}\Big)(-i\omega)^{\alpha}\widehat{f}(\omega)
\nonumber\\
&=&\Big(e^{\gamma z}-e^{pz}
\big(\frac{1-e^{-z}}{z}\big)^{\alpha}\Big)(-i\omega)^{\alpha}\widehat{f}(\omega)
\nonumber\\
&:=&W_{\alpha,p}(z)(-i\omega)^{\alpha}\widehat{f}(\omega),
\end{eqnarray}
where $z=-i\omega h$. Since $W_{\alpha,p}(z)$ is analytic in some neighborhood of the origin, we have the power
series expansion $W_{\alpha,p}(z)=\sum_{l=0}^{\infty}a_l z^l$, which converges absolutely for all $|z|\leq R$ for some $R > 0$.
Note that $a_0=0$.
So
\begin{eqnarray}\label{equation2.21*3}
&&\mathcal{F}\big[\,_{a}D_{x}^{\alpha}f(x+\gamma h)-\delta_{h,p}^\alpha f(x)\big](\omega)
\nonumber\\
&=&\sum_{l=0}^{n-1}a_l (-i\omega )^{\alpha+l}h^l\widehat{f}(\omega)+\widehat{\varphi}(\omega,h),
\end{eqnarray}
where
\begin{eqnarray*}
|\widehat{\varphi}(\omega,h)|
=\big|\big(W_{\alpha,p}(-i\omega h)-\sum_{l=0}^{n-1}a_l (-i\omega h)^l\big)\big|\cdot\big| (-i\omega)^{\alpha}\widehat{f}(\omega)\big|.
\end{eqnarray*}
We next show that there exists a constant $C_1 > 0$ such that
\begin{eqnarray}\label{equation2.21*4}
|\widehat{\varphi}(\omega,h)|
\leq C_1h^n \big|\mathcal{F}\big[\,_{a}D_{x}^{\alpha+n}f(x)\big](\omega)\big|
\end{eqnarray}
uniformly for $\omega h\in \mathbb{R}$. In fact, when $|\omega h|\leq R$, we have
\begin{eqnarray*}
&&\big|\big(W_{\alpha,p}(-i\omega h)-\sum_{l=0}^{n-1}a_l (-i\omega h)^l\big) \big|\cdot\big| (-i\omega)^{\alpha}\widehat{f}(\omega)\big|
\nonumber\\
&=&\big|\sum_{l=n}^{\infty}a_l(-i\omega h)^{l-n}\cdot h^n(-i\omega)^{\alpha+n}\widehat{f}(\omega)\big|
\nonumber\\
&=& \big|\sum_{l=n}^{\infty}a_l(-i\omega h)^{l-n}\big| \cdot h^n\big |(-i\omega)^{\alpha+n}\widehat{f}(\omega)\big|
\nonumber\\
&\leq&C_2 h^n\big|(-i\omega)^{\alpha+n}\widehat{f}(\omega)\big|
\nonumber\\
&=&C_2 h^n \big|\mathcal{F}\big[\,_{a}D_{x}^{\alpha+n}f(x)\big](\omega)\big|;
\end{eqnarray*}
where $C_2=R^{-n}\sum_{l=0}^{\infty}|a_l| R^l<\infty$;
and when $|\omega h|>R$, we have
\begin{eqnarray*}
&&\big|W_{\alpha,p}(-i\omega h)(-i\omega)^{\alpha}\widehat{f}(\omega)\big|
\nonumber\\
&=&\big|\Big(e^{-i\omega\gamma h}-e^{-i\omega p h}
\big(\frac{1-e^{i\omega h}}{-i\omega h}\big)^{\alpha}\Big) \cdot (-i\omega h)^{-n}\cdot
h^n (-i\omega)^{\alpha+n}\widehat{f}(\omega)\big|
\nonumber\\
&=& \big|\Big(e^{-i\omega\gamma h}-e^{-i\omega p h}
\big(\frac{1-e^{i\omega h}}{-i\omega h}\big)^{\alpha}\Big)\big|\cdot \big|(-i\omega h)^{-n}\big|\cdot
h^n \big|(-i\omega)^{\alpha+n}\widehat{f}(\omega)\big|
\nonumber\\
&\leq&C_3 h^n|(-i\omega)^{\alpha+n}\widehat{f}(\omega)|
\nonumber\\
&=&C_3 h^n \big|\mathcal{F}\big[\,_{a}D_{x}^{\alpha+n}f(x)\big](\omega)\big|,
\end{eqnarray*}
where $C_3=\big(1+\frac{2^\alpha}{R^\alpha}\big)\frac{1}{R^n}<\infty$; and
\begin{eqnarray*}
&&|\sum_{l=0}^{n-1}a_l (-i\omega h)^l\cdot(-i\omega)^{\alpha}\widehat{f}(\omega)|
\nonumber\\
&=&\big|\sum_{l=0}^{n-1}a_l (-i\omega h)^{l-n} \cdot
h^n (-i\omega)^{\alpha+n}\widehat{f}(\omega)\big|
\nonumber\\
&\leq&C_4 h^n|(-i\omega)^{\alpha+n}\widehat{f}(\omega)|
\nonumber\\
&=&C_4 h^n \big|\mathcal{F}\big[\,_{a}D_{x}^{\alpha+n}f(x)\big](\omega)\big|,
\end{eqnarray*}
where $C_4=\sum_{l=0}^{n-1}|a_l| R^{l-n}<\infty$.
Now if we set $C_1=\max\{C_2,C_3+C_4\}$, then it follows that (\ref{equation2.21*4}) holds for all
$\omega h\in \mathbb{R}$. Thus $\widehat{\varphi}(\omega,h)\in L^1(\mathbb{R})$ and
$\mathcal{F}\big[\,_{a}D_{x}^{\alpha}f(x+\gamma h)-\delta_{h,p}^\alpha f(x)\big](\omega)\in L^1(\mathbb{R})$.

Performing the inverse Fourier transform on (\ref{equation2.21*3}) leads to
\begin{equation}
\,_{a}D_{x}^{\alpha}f(x+\gamma h)-\delta_{h,p}^\alpha f(x)
=\sum_{l=1}^{n-1}a_l(\gamma,p)\,_{a}D_{x}^{\alpha+l}f(x)h^l+\varphi(x,h),
\end{equation}
where
\begin{equation}
|\varphi(x,h)|=\big|\frac{1}{2 \pi}\int_{\mathbb{R}}e^{-i\omega x}\widehat{\varphi}(\omega,h)d\omega\big|\leq
\frac{1}{2 \pi}\int_{\mathbb{R}}|\widehat{\varphi}(\omega,h)|d\omega\leq Ch^n.
\end{equation}
The proof is completed.
\end{proof}

By simple calculation, we have the coefficients
\begin{equation}\label{equation2.22}
\left\{ \begin{array}{lll}
a_0(\gamma,p)&=&0;\\
a_1(\gamma,p)&=&\gamma-(p-\frac{\alpha}{2});\\
a_2(\gamma,p)&=&-\frac{\alpha}{24}-\frac{1}{2}(p-\frac{\alpha}{2})^2+\frac{\gamma^2}{2};\\
a_3(\gamma,p)&=&\frac{\alpha^3+\alpha^2}{48}-\frac{p(3\alpha^2+\alpha)}{24}
+\frac{p^2\alpha}{4}+\frac{\gamma^3-p^3}{6}.
\end{array} \right.
\end{equation}

Intuitively, in Lemma \ref{lemma2.21} it seems unreasonable to require the regularity of the performed function at the right end point when analyzing its left fractional derivative. Next, we show that this requirement can be dropped. 
\begin{theorem}\label{the2.21}
Let $m-1\leq \alpha<m$, $m,n\in \mathbb{N}^{+}$, $f(x)\in C^{n+m-1}[a,b]$, $D^{n+m}f(x)$ $ \in L^1[a,b]$, and $D^{k}f(a)=0,~k=0,1,\cdots,n+m-1$.
Then for any integer $p$ and a real parameter $\gamma$, there exists
\begin{equation}
\,_{a}D_{x}^{\alpha}f(x+\gamma h)-\delta_{h,p}^\alpha f(x)
=\sum_{l=1}^{n-1}a_l(\gamma,p)\,_{a}D_{x}^{\alpha+l}f(x)h^l+O(h^n)
\end{equation}
for $x+\gamma h\in[a,b]$, where $a_l(\gamma,p)$ are the coefficients of the power series of the function
$W_{\alpha,p}(z)=\bigg(e^{\gamma z}-e^{pz}\big(\frac{1-e^{-z}}{z}\big)^{\alpha}\bigg)$, i.e.,
$W_{\alpha,p}(z)=\sum_{l=0}^{\infty}a_l z^l$.
\end{theorem}
\begin{proof}

Firstly, we can always construct a function $y(x)$, which satisfies:
$y(x)\in C^{n+m-1}[a,b']$, $D^{n+m}y(x) \in L^1[a,b']$, and $D^{k}y(a)=D^{k}y(b')=0$, $k=0,1,\cdots,$ $n+m-1$ for some $b'\geq b$;
and $y(x)=f(x)$ for $x\in[a,b]$.
In this way,
\begin{equation}
\,_{a}D_{x}^{\alpha}y(x+\gamma h)-\delta_{h,p}^\alpha y(x)
=\,_{a}D_{x}^{\alpha}f(x+\gamma h)-\delta_{h,p}^\alpha f(x)~~ \textrm{for~} x+\gamma h\in[a,b].
\end{equation}

By Lemma \ref{lemma2.21}, we have
\begin{equation}
\,_{a}D_{x}^{\alpha}y(x+\gamma h)-\delta_{h,p}^\alpha y(x)
=\sum_{l=1}^{n-1}a_l(\gamma,p)\,_{a}D_{x}^{\alpha+l}y(x)h^l+O(h^n)
\end{equation}
 for $x+\gamma h\in[a,b']$; so
\begin{equation}
\,_{a}D_{x}^{\alpha}f(x+\gamma h)-\delta_{h,p}^\alpha f(x)
=\sum_{l=1}^{n-1}a_l(\gamma,p)\,_{a}D_{x}^{\alpha+l}f(x)h^l+O(h^n)
\end{equation}
 for $x+\gamma h\in[a,b]$, which completes the proof.
\end{proof}

From Theorem \ref{the2.21}, we can easily derive the following theorem on the second order quasi-compact approximations in bounded domain and the regularity requirements for the performed functions.
\begin{theorem}\label{the2.23}
Let $f(x)\in C^3[a,b]$, $D^4 f(x) \in L^1[a,b]$, and $D^{k}f(a)=0,~k=0,1,2,3$. Then the quasi-compact approximations corresponding to  (\ref{equation2.111}), (\ref{equation2.114}), and (\ref{equation2.120}) have second order accuracy and
share a genetic form
\begin{eqnarray}\label{equation2.24}
&&c_{-1}\,_{a}D_{x}^{\alpha}f(x-h)+c_{0}\,_{a}D_{x}^{\alpha}f(x)+c_{1}\,_{a}D_{x}^{\alpha}f(x+h)
\nonumber\\
&=&d_{p}\delta^{\alpha}_{h,p}f(x)+d_q\delta^{\alpha}_{h,q}f(x)+O(h^{2}),
\end{eqnarray}
where $(p,q)=(-1,1)$ or $(0,1)$ or $(-1,0)$, and
\begin{equation}\label{equation2.25}
c_{-1}+c_{0}+c_{1}=d_{0}+d_{1}=1.
\end{equation}
Moreover, if $f(x)\in C^4[a,b]$, $D^5f(x) \in L^1[a,b]$, and $D^{k}f(a)=0,~k=0,\cdots,4$,
then the quasi-compact approximations corresponding to (\ref{equation2.111}), (\ref{equation2.114}), and (\ref{equation2.120})
share the following form
\begin{eqnarray}\label{equation2.26}
&&c_{-1}\,_{a}D_{x}^{\alpha}f(x-h)+c_{0}\,_{a}D_{x}^{\alpha}f(x)+c_{1}\,_{a}D_{x}^{\alpha}f(x+h)
\nonumber\\
&=&d_{p}\delta^{\alpha}_{h,p}f(x)+d_q\delta^{\alpha}_{h,q}f(x)+e_2\,_{a}D_{x}^{\alpha+2}f(x)h^2+O(h^{3}),
\end{eqnarray}
where $(p,q)=(-1,1)$ or $(0,1)$ or $(-1,0)$, and
\begin{eqnarray}\label{equation2.27}
&&c_{-1}+c_{0}+c_{1}=d_{0}+d_{1}=1;
\nonumber\\
e_{2}&=&c_{-1}\cdot a_{2}(-1,p)+(d_{p}-c_{-1})\cdot a_{2}(0,p)
\nonumber\\
&&+(d_{q}-c_{1})\cdot a_{2}(0,q)+c_{1}\cdot a_{2}(1,q).
\end{eqnarray}
\end{theorem}
%


Now for the convenience of the discussions in the next (sub)sections, we list several specific asymptotic expansions for the quasi-compact approximations given above.
\begin{proposition}
Let $f(x)\in C^4[a,b]$, $D^5f(x) \in L^1[a,b]$, and $D^{k}f(a)=0,~k=0,\cdots,4$.
Then
\begin{eqnarray}\label{equation2.29}
&&(1-\beta)\,_{a}D_{x}^{\alpha}f(x)+\beta\,_{a}D_{x}^{\alpha}f(x+h)-\delta^{\alpha}_{h,1}f(x)
\nonumber\\
&=&\frac{(1-\beta)(6\beta-1)}{12}\,_{a}D_{x}^{\alpha+2}f(x)h^2+O(h^3),
\end{eqnarray}

\begin{eqnarray}\label{equation2.28}
&&\frac{-\beta(1-\beta)}{2}\,_{a}D_{x}^{\alpha}f(x- h)+(1-\beta^2)\,_{a}D_{x}^{\alpha}f(x)+
\frac{\beta(1+\beta)}{2}\,_{a}D_{x}^{\alpha}f(x+h)
\nonumber\\
&&-\delta^{\alpha}_{h,1}f(x)
\nonumber\\
&=&-\frac{1-\beta}{12}\,_{a}D_{x}^{\alpha+2}f(x)h^2+O(h^3),
\end{eqnarray}
and

\begin{eqnarray}\label{equation2.210}
&&\,_{a}D_{x}^{\alpha}f(x)-\big[\beta\delta^{\alpha}_{h,0}f(x)+(1-\beta)\delta^{\alpha}_{h,1}f(x)\big]
\nonumber\\
&=&-\frac{(1-\beta)(6\beta+1)}{12}\,_{a}D_{x}^{\alpha+2}f(x)h^2+O(h^3),
\end{eqnarray}

\begin{eqnarray}\label{equation2.211}
&&\beta(1-\beta)\,_{a}D_{x}^{\alpha}f(x-h)+[\beta^2+(1-\beta)^2]\,_{a}D_{x}^{\alpha}f(x)
\nonumber\\
&&+\beta(1-\beta)\,_{a}D_{x}^{\alpha}f(x+h)
-\big[\beta\delta^{\alpha}_{h,0}f(x)+(1-\beta)\delta^{\alpha}_{h,1}f(x)\big]
\nonumber\\
&=&\frac{(1-\beta)(6\beta-1)}{12}\,_{a}D_{x}^{\alpha+2}f(x)h^2+O(h^3).
\end{eqnarray}

\end{proposition}

\subsection{Derivation of a series of higher order approximations}\label{subsec:2.3}

In this subsection, we focus on stating a basic strategy to derive any desired high order approximation. As an illustrating example, several third order approximations are specifically deduced.  Combining any two of the second order quasi-compact approximations listed above can lead to a third order approximation; so the number of different third approximations is much greater than the second order's one. Along this direction, letting  $f(x)\in C^{l+3}[a,b]$, $D^{l+4}f(x) \in L^1[a,b]$, and $D^{k}f(a)=0,~k=0,\cdots,l+3$,
from any two of the quasi-compact low order approximations, say
\begin{eqnarray}\label{equation2.34}
&&c_{-1,i}\,_{a}D_{x}^{\alpha}f(x-h)+c_{0,i}\,_{a}D_{x}^{\alpha}f(x)+c_{1,i}\,_{a}D_{x}^{\alpha}f(x+h)
\nonumber\\
&=&d_{-1,i}\delta^{\alpha}_{h,-1}f(x)+d_{0,i}\delta^{\alpha}_{h,0}f(x)+d_{1,i}\delta^{\alpha}_{h,1}f(x)+e_{l,i}h^l+O(h^{l+1}),
\end{eqnarray}
and
\begin{eqnarray}\label{equation2.35}
&&c_{-1,j}\,_{a}D_{x}^{\alpha}f(x-h)+c_{0,j}\,_{a}D_{x}^{\alpha}f(x)+c_{1,j}\,_{a}D_{x}^{\alpha}f(x+h)
\nonumber\\
&=&d_{-1,j}\delta^{\alpha}_{h,-1}f(x)+d_{0,j}\delta^{\alpha}_{h,0}f(x)+d_{1,j}\delta^{\alpha}_{h,1}f(x)+e_{l,j}h^l+O(h^{l+1}),
\end{eqnarray}
we can get a higher order one
\begin{eqnarray}\label{equation2.36}
&&\tilde{c}_{-1}\,_{a}D_{x}^{\alpha}f(x-h)+\tilde{c}_{0}\,_{a}D_{x}^{\alpha}f(x)+\tilde{c}_{1}\,_{a}D_{x}^{\alpha}f(x+h)
\nonumber\\
&=&\tilde{d}_{-1}\delta^{\alpha}_{h,-1}f(x)+\tilde{d}_{0}\delta^{\alpha}_{h,0}f(x)+\tilde{d}_1\delta^{\alpha}_{h,1}f(x)+\tilde{e}_{l+1}(h^{l+1})+O(h^{l+2}),
\end{eqnarray}
where
\begin{equation}\label{equation2.37}
\left\{ \begin{array}{lll}
\tilde{c}_{-1}&=&e_{j,l}c_{-1,i}-e_{i,l}c_{-1,j};\\
\tilde{c}_{0}&=&e_{j,l}c_{0,i}-e_{i,l}c_{0,j};\\
\tilde{c}_{1}&=&e_{j,l}c_{1,i}-e_{i,l}c_{1,j};\\
\tilde{d}_{-1}&=&e_{j,l}d_{-1,i}-e_{i,l}d_{-1,j};\\
\tilde{d}_{0}&=&e_{j,l}d_{0,i}-e_{i,l}d_{0,j};\\
\tilde{d}_{1}&=&e_{j,l}d_{1,i}-e_{i,l}d_{1,j};\\
\tilde{c}_{-1}&+&\tilde{c}_{0}+\tilde{c}_{1}=\tilde{d}_{-1}+\tilde{d}_{0}+\tilde{d}_{1};\\
\tilde{e}_{l+1}&=&\tilde{c}_{-1}\cdot a_{l+1}(-1,-1)+(\tilde{d}_{-1}-\tilde{c}_{-1})\cdot a_{l+1}(0,-1)\\
&&+d_0\cdot a_{l+1}(0,0)+(\tilde{d}_{1}-\tilde{c}_{1})\cdot a_{l+1}(0,1)+\tilde{c}_{1}\cdot a_{l+1}(1,1).
\end{array} \right.
\end{equation}

Now from the second order quasi-compact approximations (\ref{equation2.29})-(\ref{equation2.211}), we deduce and list the following four different third order quasi-compact approximations. Under the assumptions: $f(x)\in C^4[a,b]$, $D^5f(x) \in L^1[a,b]$, and $D^{k}f(a)=0,~k=0,\cdots,4$;
by calculating: $\frac{1-\beta}{12}\times(\ref{equation2.29})+\frac{(1-\beta)(6\beta-1)}{12}\times(\ref{equation2.28})$, we get
\begin{eqnarray}\label{equation2.38}
&&-\frac{\beta(1-\beta)^2(6\beta-1)}{24}\,_{a}D_{x}^{\alpha}f(x-h)+
\frac{\beta(1-\beta)^2(6\beta+5)}{12}\,_{a}D_{x}^{\alpha}f(x)
\nonumber\\
&&+\frac{\beta(1-\beta)(2\beta+1)(3\beta+1)}{24}\,_{a}D_{x}^{\alpha}f(x+h)
\nonumber\\
&=&\frac{\beta(1-\beta)}{2}\delta^{\alpha}_{h,1}f(x)+O(h^3);
\end{eqnarray}
from $\frac{(1-\beta)(6\beta+1)}{12}\times(\ref{equation2.29})+\frac{(1-\beta)(6\beta-1)}{12}\times(\ref{equation2.210})$, we have
\begin{eqnarray}\label{equation2.311}
&&\frac{\beta(1-\beta)(11-6\beta)}{12}\,_{a}D_{x}^{\alpha}f(x)+
\frac{\beta(1-\beta)(6\beta+1)}{12}\,_{a}D_{x}^{\alpha}f(x+h)
\nonumber\\
&=&\frac{\beta(1-\beta)(6\beta-1)}{12}\delta^{\alpha}_{h,0}f(x)+\frac{\beta(1-\beta)(13-6\beta)}{12}\delta^{\alpha}_{h,1}f(x)+O(h^3);
\end{eqnarray}
according to $-\frac{(1-\beta)(6\beta-1)}{12}\times(\ref{equation2.29})+\frac{(1-\beta)(6\beta-1)}{12}\times(\ref{equation2.211})$, there exists
\begin{eqnarray}\label{equation2.312}
&&-\frac{\beta(1-\beta)^2(6\beta-1)}{12}\,_{a}D_{x}^{\alpha}f(x-h)-
\frac{\beta(1-\beta)(6\beta-1)(2\beta-1)}{12}\,_{a}D_{x}^{\alpha}f(x)
\nonumber\\
&&+\frac{\beta^2(1-\beta)(6\beta-1)}{12}\,_{a}D_{x}^{\alpha}f(x+h)
\nonumber\\
&=&\frac{\beta(1-\beta)(6\beta-1)}{12}\big[-\delta^{\alpha}_{h,0}f(x)+\delta^{\alpha}_{h,1}f(x)\big]+O(h^3);
\end{eqnarray}
by calculating:
$\frac{(1-\beta)(6\beta-1)}{12}\times(\ref{equation2.210})+\frac{(1-\beta)(6\beta+1)}{12}\times(\ref{equation2.211})$, there is
\begin{eqnarray}\label{equation2.313}
&&\frac{\beta(1-\beta)^2(6\beta+1)}{12}\,_{a}D_{x}^{\alpha}f(x-h)+
\frac{\beta(1-\beta)(6\beta^2-5\beta+5)}{6}\,_{a}D_{x}^{\alpha}f(x)
\nonumber\\
&&+\frac{\beta(1-\beta)^2(6\beta+1)}{12}\,_{a}D_{x}^{\alpha}f(x+h)
\nonumber\\
&=&\beta(1-\beta)\big[\beta\delta^{\alpha}_{h,0}f(x)+(1-\beta)\delta^{\alpha}_{h,1}f(x)\big]+O(h^3).
\end{eqnarray}

\begin{remark}\label{remark2.31}
It can be noted that if $\alpha=2$ and $\beta=1-\alpha/2=0$, then all the stable (which means that the derived scheme is stable when the discretization is used to solve space fractional diffusion equation; see next section) second order quasi-compact approximations
reduce to the standard centered difference operator
\begin{equation}
D^2f(x)=\frac{f(x-h)-2f(x)+f(x+h)}{h^2}+O(h^2).
\end{equation}

For the third order approximations firstly dividing $\beta$ in both sides of them and then letting $\beta \rightarrow 0$, we get the third order quasi-compact approximations of some linear combinations of the classical second order derivatives.
\end{remark}

%
%

\section{Stable high order schemes for fractional diffusion problems}\label{sec:3}

Based on the high order approximations to the linear combinations of the Riemann-Liouville space fractional derivatives, we develop a series of high order quasi-compact Crank-Nicolson type scheme for the problem (\ref{equation1.1}). Then we perform the detailed stability and convergence analysis for several  derived schemes. And some of the error estimates are also discussed.

\subsection{Derivation of the general numerical scheme}\label{subsec3.1}

We partition the interval $[x_L,x_R]$ into an uniform mesh with the space stepsize $h=(x_R-x_L)/N$ and the time steplength
$\tau=T/M$, where $N,~M$ are two positive integers. The sets of mesh points are denoted by $x_i=ih$ for $1\leq i\leq N-1$
and $t_n=n\tau$ for $0\leq n \leq M$. Let $t_{n+1/2}=(t_n+t_{n+1})/2$ for $0\leq n \leq M-1$. And the
following notations are used in the sections below
\begin{equation}\label{equation3.11}
u_i^n=u(x_i,t_n),~~~~~f_i^{n+1/2}=f(x_i,t_{n+1/2}).
\end{equation}

For the space discretization, by combining the general quasi-compact form (\ref{equation2.34}) with (\ref{equation1.1})
($c_{-1}$ must equal to $c_1$, if $K_1K_2\neq 0$), it yields that, for a fixed node $x_i$, $1\leq i \leq N-1,$
\begin{eqnarray}\label{equation3.12}
&&c_{-1}\frac{\partial u(x_{i-1},t)}{\partial t}+c_{0}\frac{\partial u(x_i,t)}{\partial t}
+c_{1}\frac{\partial u(x_{i+1},t)}{\partial t}
\nonumber\\
&=&K_1\big[d_{-1}\delta^{\alpha}_{h,-1}u(x_i,t)+d_0\delta^{\alpha}_{h,0}u(x_i,t)+d_1\delta^{\alpha}_{h,1}u(x_i,t)\big]
\nonumber\\
&&+K_2\big[d_{-1}\sigma^{\alpha}_{h,-1}u(x_i,t)+d_0\sigma^{\alpha}_{h,0}u(x_i,t)+d_1\sigma^{\alpha}_{h,1}u(x_i,t)\big]
\nonumber\\
&&~+c_{-1}f(x_{i-1},t)+c_{0}f(x_{i},t)+c_{1}f(x_{i+1},t)+O(h^l),
\end{eqnarray}
where $l=2,3,4,\cdots$.
Using the Crank-Nicolson (CN) technique to discretize the time derivative of (\ref{equation1.1}) leads to
\begin{eqnarray}\label{equation3.13}
&&c_{-1}\frac{u_{i-1}^{n+1}-u_{i-1}^{n}}{\tau}+c_{0}\frac{u_{i}^{n+1}-u_{i}^{n}}{\tau}
+c_{1}\frac{u_{i+1}^{n+1}-u_{i+1}^{n}}{\tau}
\nonumber\\
&=&K_1\big(d_{-1}\delta^{\alpha}_{h,-1}u(x_i,t)+d_0\delta^{\alpha}_{h,0}+d_1\delta^{\alpha}_{h,1}\big)\frac{u_i^{n}+u_i^{n+1}}{2}
\nonumber\\
&&+K_2\big(d_{-1}\sigma^{\alpha}_{h,-1}u(x_i,t)+d_0\sigma^{\alpha}_{h,0}+d_1\sigma^{\alpha}_{h,1}\big)\frac{u_i^{n}+u_i^{n+1}}{2}
\nonumber\\
&&~+c_{-1}f_{i-1}^{n+1/2}+c_{0}f_{i}^{n+1/2}+c_{1}f_{i+1}^{n+1/2}+O(h^l+\tau^2),
\end{eqnarray}
i.e.,
\begin{eqnarray}\label{equation3.14}
&&c_{-1}u_{i-1}^{n+1}+c_{0}u_{i}^{n+1}+c_{1}u_{i+1}^{n+1}
\nonumber\\
&&-\frac{K_1\tau}{2h^{\alpha}}\sum_{k=0}^{i+1} w_k^{(\alpha)} u_{i-k+1}^{n+1}
-\frac{K_2\tau}{2h^{\alpha}}\sum_{k=0}^{i+1} w_k^{(\alpha)} u_{i+k-1}^{n+1}
\nonumber\\
&=&c_{-1}u_{i-1}^{n}+c_{0}u_{i}^{n}+c_{1}u_{i+1}^{n}
+\frac{K_1\tau}{2h^{\alpha}}\sum_{k=0}^{i+1} w_k^{(\alpha)} u_{i-k+1}^{n}
+\frac{K_2\tau}{2h^{\alpha}}\sum_{k=0}^{i+1} w_k^{(\alpha)} u_{i+k-1}^{n}
\nonumber\\
&&~+\tau\big(c_{-1}f_{i-1}^{n+1/2}+c_{0}f_{i}^{n+1/2}+c_{1}f_{i+1}^{n+1/2}\big)+O(\tau h^l+\tau^3),
\end{eqnarray}
where
\begin{equation}\label{equation3.15}
\left\{ \begin{array}{lll}
w_0^{(\alpha)}&=&d_1 g_0^{(\alpha)},\\
w_1^{(\alpha)}&=&d_0 g_0^{(\alpha)}+d_1 g_1^{(\alpha)},\\
w_k^{(\alpha)}&=&d_{-1} g_{k-2}^{(\alpha)}+d_0 g_{k-1}^{(\alpha)}+d_1 g_k^{(\alpha)},~k\geq 2.
\end{array} \right.
\end{equation}

Denoting $U_{i}^{n}$ as the numerical approximation of $u_{i}^{n}$, the so-called
CN-quasi-compact scheme for (\ref{equation1.1}) is obtained as
\begin{eqnarray}\label{equation3.16}
&&c_{-1}U_{i-1}^{n+1}+c_{0}U_{i}^{n+1}+c_{1}U_{i+1}^{n+1}
\nonumber\\
&&-\frac{K_1\tau}{2h^{\alpha}}\sum_{k=0}^{i+1} w_k^{(\alpha)} U_{i-k+1}^{n+1}
-\frac{K_2\tau}{2h^{\alpha}}\sum_{k=0}^{i+1} w_k^{(\alpha)} U_{i+k-1}^{n+1}
\nonumber\\
&=&c_{-1}U_{i-1}^{n}+c_{0}U_{i}^{n}+c_{1}U_{i+1}^{n}
+\frac{K_1\tau}{2h^{\alpha}}\sum_{k=0}^{i+1} w_k^{(\alpha)} U_{i-k+1}^{n}
\nonumber\\
&&+\frac{K_2\tau}{2h^{\alpha}}\sum_{k=0}^{i+1} w_k^{(\alpha)} U_{i+k-1}^{n}
+\tau\big(c_{-1}f_{i-1}^{n+1/2}+c_{0}f_{i}^{n+1/2}+c_{1}f_{i+1}^{n+1/2}\big).
\end{eqnarray}

Defining the column vectors
\begin{eqnarray*}\label{equation3.17}
\textbf{U}^n&=&\big(U_{1}^n,U_{2}^n,\cdots,U_{N-1}^{n}\big)^T,\\
\widetilde{\textbf{F}}^n&=&\big(f_0^{n+1/2},f_{1}^{n+1/2},\cdots,f_{N}^{n+1/2}\big)^T,\\
\textbf{F}^n&=&c_{-1}\widetilde{\textbf{F}}^n(0:N-2)+c_{0}\widetilde{\textbf{F}}^n(1:N-1)+c_{1}\widetilde{\textbf{F}}^n(2:N),
\end{eqnarray*}
Eq. (\ref{equation3.16}) can be recast in matrix form as
\begin{equation}\label{equation3.18}
\big(\textbf{T}-\frac{\tau}{2h^{\alpha}}(K_1 \textbf{A}+K_2 \textbf{A}^{T})\big) \textbf{U}^{n+1}
=\big(\textbf{T}+\frac{\tau}{2h^{\alpha}}(K_1 \textbf{A}+K_2 \textbf{A}^{T})\big) \textbf{U}^{n}+\tau \textbf{F}^n+\textbf{H}^{n},
\end{equation}
where
\begin{equation}\label{equation3.19}
\textbf{T}=\left[ \begin{array}{ccccc}
c_{0}   &   c_{1}  &       &&           \\
c_{-1}  &   c_{0}  &  c_{1}   &&          \\
        &\ddots  &  \ddots    &   \ddots  &\\
        &        &  c_{-1}  &   c_{0}   &  c_{1} \\
        &        &          &   c_{-1}  &  c_{0}
\end{array}
\right],
\end{equation}

\begin{equation}\label{equation3.110}
\textbf{A}=\left [ \begin{array}{cccccc}
w_{1}^{(\alpha)}   &   w_{0}^{(\alpha)}  &       &&&            \\
w_{2}^{(\alpha)}  &   w_{1}^{(\alpha)}  &  w_{0}^{(\alpha)} &&&         \\
        &\ddots  &  \ddots    &   \ddots && \\
w_{N-2}^{(\alpha)} & w_{N-3}^{(\alpha)}& \cdots  &  w_{2}^{(\alpha)} &w_{1}^{(\alpha)}   &  w_{0}^{(\alpha)} \\
w_{N-1}^{(\alpha)} & w_{N-2}^{(\alpha)} & \cdots  & w_{3}^{(\alpha)}& w_{2}^{(\alpha)}  &  w_{1}^{(\alpha)}
\end{array}
\right ],
\end{equation}
are Toeplitz matrices, and
\begin{eqnarray}\label{equation3.111}
\textbf{H}^n&=&\left [ \begin{array}{c}
c_{-1}\\
0\\
\vdots\\
0
\end{array}\right ](\textbf{U}_0^n-\textbf{U}_0^{n+1})
+
\left [ \begin{array}{c}
0\\
0\\
\vdots\\
c_{1}
\end{array}\right ](\textbf{U}_N^{n}-\textbf{U}_N^{n+1})
\nonumber\\
&&+\frac{\tau}{2h^{\alpha}}\left [ \begin{array}{c}
K_1 w_2^{(\alpha)}+K_2w_0^{(\alpha)}\\
K_1w_3^{(\alpha)}\\
\vdots\\
K_1w_{N-1}^{(\alpha)}\\
K_1w_{N}^{(\alpha)}
\end{array}\right ](\textbf{U}_0^n+\textbf{U}_0^{n+1})
\nonumber\\
&&+\frac{\tau}{2h^{\alpha}}\left [ \begin{array}{c}
K_2 w_N^{(\alpha)}\\
K_2w_{N-1}^{(\alpha)}\\
\vdots\\
K_2w_{3}^{(\alpha)}\\
K_1w_{0}^{(\alpha)}+K_2w_{2}^{(\alpha)}
\end{array}\right ](\textbf{U}_N^{n}+\textbf{U}_N^{n+1})
\end{eqnarray}
is the vector obtained by imposing the boundary conditions.

\begin{remark}\label{remark3.11}
If $K_1K_2 \neq 0$ in (\ref{equation1.1}), $c_{-1}$ must equal to $c_1$, like the ones in second order approximations
(\ref{equation2.210})-(\ref{equation2.211}) and third order approximation (\ref{equation2.313}).
In these cases, the matrices $\textbf{T}$ in the corresponding matrix forms (\ref{equation3.18}) are symmetric.
Generally speaking, if $c_{-1}\neq c_1$ it is hard/impossible to get a high order scheme for (\ref{equation1.1}) with $K_1K_2\neq 0$. However, for second order approximations, as do in \cite{Nasir:13}, we can firstly plug (\ref{equation2.15}) into
the two-sided problem (\ref{equation1.1}), and then expand $\frac{\partial u(x+\beta h,t)}{\partial t}$ in Taylor's series w.r.t $x$,
similar to the way of getting (\ref{equation2.116}), (\ref{equation2.117}) and (\ref{equation2.120}).
That is,
\begin{eqnarray}\label{equation3.112}
&&-\beta\frac{ u(x-h,t+\tau)-u(x-h,t)}{\tau}+(1+\beta)\frac{u(x,t+
\tau)-u(x,t)}{\tau}
\nonumber\\
&=&K_1\delta^{\alpha}_{h,1}u(x,t+\frac{\tau}{2})+K_2\sigma^{\alpha}_{h,1}u(x,t+\frac{\tau}{2})
\nonumber\\
&&+f(x+\beta h,t+\frac{\tau}{2})+O(h^2+\tau^2),
\end{eqnarray}

\begin{eqnarray}\label{equation3.113}
&&(1-\beta)\frac{u(x,t+
\tau)-u(x,t)}{\tau}+\beta\frac{u(x+h,t+\tau)-u(x+h,t)}{\tau}
\nonumber\\
&=&K_1\delta^{\alpha}_{h,1}u(x,t+\frac{\tau}{2})+K_2\sigma^{\alpha}_{h,1}u(x,t+\frac{\tau}{2})
\nonumber\\
&&+f(x+\beta h,t+\frac{\tau}{2}))+O(h^2+\tau^2),
\end{eqnarray}

\begin{eqnarray}\label{equation3.113}
&&-\frac{\beta(1-\beta)}{2}\cdot\frac{ u(x-h,t+\tau)-u(x-h,t)}{\tau}+(1-\beta^2)\frac{u(x,t+
\tau)-u(x,t)}{\tau}
\nonumber\\
&&+\frac{\beta(1+\beta)}{2}\cdot\frac{u(x+h,t+\tau)-u(x+h,t)}{\tau}
\nonumber\\
&=&K_1\delta^{\alpha}_{h,1}u(x,t+\frac{\tau}{2})+K_2\sigma^{\alpha}_{h,1}u(x,t+\frac{\tau}{2})
\nonumber\\
&&+f(x+\beta h,t+\frac{\tau}{2})+O(h^2+\tau^2).
\end{eqnarray}
\end{remark}


\subsection{Properties of the matrices for the derived schemes}\label{subsec3.2}

In this subsection we prove that the differential matrix $\textbf{A}$ in (\ref{equation3.110}) is negative definite for some of the second order and third order quasi-compact schemes.  First, let us list some preliminary results.

\begin{definition} (\cite{Quarteroni:07})\label{definition3.21}
A matrix $\textbf{A} \in \mathbb{R}^{n\times n}$ is said to be positive definite in $\mathbb{R}^{n}$,
if $\textbf{x}^T\textbf{A}\textbf{x}>0$ for all $\textbf{x} \in \mathbb{R}^{n}$, $\textbf{x}\neq \textbf{0}$.
\end{definition}

\begin{lemma} (\cite{Quarteroni:07})\label{lemma3.21}
A real matrix $\textbf{A}$ of order $n$ is positive definite  if and only if
its symmetric part $\textbf{H}=\frac{\textbf{A}+\textbf{A}^T}{2}$ is positive definite.
Let $\textbf{H} \in \mathbb{R}^{n\times n}$ be symmetric. Then $\textbf{H}$ is positive definite
if and only if the eigenvalues of $\textbf{H}$ are positive.
\end{lemma}

\begin{lemma} (\cite{Quarteroni:07})\label{lemma3.22}
If $\textbf{A} \in \mathbb{C}^{n \times n}$, let $\textbf{H}=\frac{\textbf{A}+\textbf{A}^H}{2}$ be the hermitian part of $\textbf{A}$.
Then for any eigenvalue $\lambda$ of  $\textbf{A}$,
the real part $Re(\lambda(\textbf{A}))$ satisfies
\begin{equation}
\lambda_{\min}(\textbf{H}) \leq Re(\lambda(\textbf{A})) \leq \lambda_{\max}(\textbf{H}),
\end{equation}
where $\lambda_{\min}(\textbf{H})$ and $\lambda_{\max}(\textbf{H})$ are the minimum and maximum of the eigenvalues of $\textbf{H}$, respectively.
\end{lemma}

\begin{definition} (\cite{Chan:07})\label{definition3.22}
Let  $n \times n$ Toeplitz  matrix  $\textbf{T}_n$ be the form:
\begin{equation}
\textbf{T}_n=\left [ \begin{array}{ccccc}
                      t_0           &      t_{-1}             &      \cdots         &       t_{2-n}       &       t_{1-n}      \\
                      t_{1}         &      t_{0}              &      t_{-1}         &      \cdots         &       t_{2-n}        \\
                     \vdots         &      t_{1}              &      t_{0}          &      \ddots         &        \vdots            \\
                     t_{n-2}        &      \cdots             &      \ddots         &      \ddots         &        t_{-1}    \\
                     t_{n-1}        &       t_{n-2}           &      \cdots         &       t_1           &        t_0
 \end{array}
 \right ],
\end{equation}
i.e., $t_{i,j}=t_{i-j}$, and $\textbf{T}_n$ is constant along its diagonals.
Assume that the diagonals $\{t_k\}_{k=-n+1}^{n-1}$ are the Fourier coefficients of a function
$f$, i.e.,
\begin{equation}
t_k=\frac{1}{2\pi}\int_{-\pi}^{\pi}f(x)e^{-ikx}dx.
\end{equation}
Then the function $f$ is called the generating function of $\textbf{T}_n$.
\end{definition}

\begin{lemma} (\cite{Chan:91,Chan:07})\label{lemma3.23} (Grenander-Szeg\"{o} theorem)
Let $\textbf{T}_n$ be given by above matrix with a generating function $f$,
where $f$ is a $2\pi$-periodic continuous real-valued functions defined on $[-\pi,\pi]$.
Let $\lambda_{\min}(\textbf{T}_n)$ and $\lambda_{\max}(\textbf{T}_n)$ denote the smallest and largest
eigenvalues of $\textbf{T}_n$, respectively. Then we have
\begin{equation}
f_{\min} \leq \lambda_{\min}(\textbf{T}_n) \leq \lambda_{\max}(\textbf{T}_n) \leq f_{\max},
\end{equation}
where $f_{\min}$ and  $f_{\max}$  is the minimum and maximum values of $f(x)$, respectively.
Moreover, if $f_{\min}< f_{\max}$, then all eigenvalues of $\textbf{T}_n$ satisfy
\begin{equation}
f_{\min} < \lambda(\textbf{T}_n) < f_{\max},
\end{equation}
for all $n>0$. In particular, if  $f_{\min}>0$, then $\textbf{T}_n$ is positive definite.
\end{lemma}

\begin{theorem}\label{the3.21}
For the second order quasi-compact schemes corresponding to (\ref{equation2.115}) with $(p,q)=(-1,1)$ or $(0,1)$, and the third order quasi-compact schemes corresponding to (\ref{equation2.38})-(\ref{equation2.313}), the $\textbf{A}$ in (\ref{equation3.110}) is negative definite.
\end{theorem}
\begin{proof}
Firstly, we consider the symmetric part of matrix $\textbf{A}$, denoted as $\textbf{H}=\frac{\textbf{A}+\textbf{A}^T}{2}$.
The generating functions of $\textbf{A}$ and $\textbf{A}^T$ are
\begin{equation}
f_A(x)=\sum_{k=0}^{\infty}w_{k}^{(\alpha)}e^{i(k-1)x},
~~~f_{A^T}(x)=\sum_{k=0}^{\infty}w_{k}^{(\alpha)}e^{-i(k-1)x},
\end{equation}
respectively. Then $f(\alpha;x)=\frac{f_A(x)+f_{A^T}(x)}{2}$ is the generating function of $\textbf{H}$,
and $f(\alpha;x)$ is a periodic continuous real-valued function on $[-\pi,\pi]$, since $f_A(x)$ and
$f_{A^T}(x)$ are mutually conjugated.

By (\ref{equation3.15}),
\begin{eqnarray}\label{equation3.21}
&&f(\alpha;x)
\nonumber\\
&=&\frac{1}{2}\big(\sum_{k=0}^{\infty}w_{k}^{(\alpha)}e^{i(k-1)x}
+\sum_{k=0}^{\infty}w_{k}^{(\alpha)}e^{-i(k-1)x}\big)
\nonumber\\
&=&\frac{1}{2}\big(d_{-1}e^{ix}\sum_{k=0}^{\infty}g_{k}^{(\alpha)}e^{ikx}+
d_{0}\sum_{k=0}^{\infty}g_{k}^{(\alpha)}e^{ikx}
+d_{1}e^{-ix}\sum_{k=0}^{\infty}g_{k}^{(\alpha)}e^{ikx}+
\nonumber\\
&&~~~~d_{-1}e^{-ix}\sum_{k=0}^{\infty}g_{k}^{(\alpha)}e^{-ikx}
+d_{0}\sum_{k=0}^{\infty}g_{k}^{(\alpha)}e^{-ikx}
+d_{1}e^{ix}\sum_{k=0}^{\infty}g_{k}^{(\alpha)}e^{-ikx}\big)
\nonumber\\
&=&\frac{1}{2}\bigg[d_{-1}\big(e^{ix}(1-e^{ix})^{\alpha}+e^{-ix}(1-e^{-ix})^{\alpha}\big)
\nonumber\\
&&~~~+d_{0}\big((1-e^{ix})^{\alpha}+(1-e^{-ix})^{\alpha}\big)
\nonumber\\
&&~~~+d_1\big(e^{-ix}(1-e^{ix})^{\alpha}+e^{ix}(1-e^{-ix})^{\alpha}\big)\bigg].
\end{eqnarray}
Since $f(\alpha;x)$ is a real-valued and even function, we just
consider its principal value on $[0,\pi]$.
From  the formulae
\begin{eqnarray*}
e^{i\theta}-e^{i\eta}&=&2i\sin(\frac{\theta-\eta}{2})e^{\frac{i(\theta+\eta)}{2}},\\
(1-e^{\pm ix})^{\alpha}&=&\left(2\sin\frac{x}{2}\right)^{\alpha}\!e^{\pm i\alpha(\frac{x}{2}-\frac{\pi}{2})},
\end{eqnarray*}
there exists
\begin{eqnarray}\label{equation3.22}
&&f(\alpha;x)
\nonumber\\
&=&(2\sin\frac{x}{2})^{\alpha}\bigg[d_{-1}\cos\big(\frac{\alpha}{2}(x-\pi)+x\big)
+d_0\cos\big(\frac{\alpha}{2}(x-\pi)\big)
\nonumber\\
&&~~~~~~~~~~~~~~~+d_1\cos\big(\frac{\alpha}{2}(x-\pi)-x\big)\bigg]
\nonumber\\
&=&(2\sin\frac{x}{2})^{\alpha}\bigg[-d_{-1}\cos\big((2-\beta)(\pi-x)\big)+
d_0\cos\big((1-\beta)(\pi-x)\big)
\nonumber\\
&&~~~~~~~~~~~~~~~-d_1\cos\big(\beta(\pi-x)\big)\bigg].
\end{eqnarray}

Since $1<\alpha\leq2$, $x\in[0,\pi]$ and $\beta=1-\frac{\alpha}{2}$, it is clear that
\begin{equation}
0\leq\beta<\frac{1}{2},~~~\frac{1}{2}<1-\beta\leq 1,
\end{equation}
and
\begin{equation}\label{equation3.23}
\sin\big(\frac{x}{2}\big)\geq 0,~~~\cos\beta\xi>0,~~~\sin\xi\geq 0,
~~~\sin\big((1-\beta)\xi\big)\geq 0,~~~\sin\beta\xi\geq 0,
\end{equation}
where $\xi=\pi-x\in[0,\pi]$.

Denoting
\begin{equation}
g(\beta):=-d_{-1}\cos\big((2-\beta)\xi\big)+d_0\cos\big((1-\beta)\xi\big)-d_1\cos\beta\xi,
\end{equation}
then
\begin{enumerate}[(i)]

\item
for the second order scheme corresponding to (\ref{equation2.111}) and (\ref{equation2.120}), where $d_{-1}=d_0=0,~d_1=1$,
\begin{equation}\label{equation3.24}
g(\beta)=-\cos\beta\xi< 0.
\end{equation}

\item
for the scheme corresponding to (\ref{equation2.114}) with $(p,q)=(-1,1)$, where $d_{-1}=\frac{\beta}{2},~d_0=0,~d_1=\frac{2-\beta}{2}$,
if $x\in[\frac{\pi}{2},\pi]$, then $\xi\in[0,\frac{\pi}{2}],~\sin2\xi>0$,
\begin{eqnarray}\label{equation3.25}
&&g(\beta)
\nonumber\\
&=&-\frac{\beta}{2}\cos\big((2-\beta)\xi\big)-(1-\frac{\beta}{2})\cos\beta\xi
\nonumber\\
&=&-\frac{\beta}{2}\big(\cos2\xi\cos\beta\xi+\sin2\xi\sin\beta\xi\big)-(1-\frac{\beta}{2})\cos\beta\xi
\nonumber\\
&=&-\cos\beta\xi\big(1-\frac{\beta}{2}(1-\cos2\xi)\big)-\frac{\beta}{2}\sin2\xi\sin\beta\xi<0;
\end{eqnarray}

if $x\in[0,\frac{\pi}{2}]$, then $\xi\in[\frac{\pi}{2},\pi],~\cos\big((\frac{1}{2}-\beta)\xi\big)\geq0,~
\sin\big((\frac{1}{2}-\beta)\xi\big)\geq 0$,
\begin{eqnarray}\label{equation3.25*1}
&&g(\beta)
\nonumber\\
&=&-\frac{\beta}{2}\cos\big((\frac{1}{2}-\beta)\xi+\frac{3}{2}\xi\big)
-(1-\frac{\beta}{2})\cos\big((\frac{1}{2}-\beta)\xi-\frac{\xi}{2}\big)
\nonumber\\
&=&-\frac{\beta}{2}\big[\cos\big((\frac{1}{2}-\beta)\xi\big)\cos\frac{3}{2}\xi
-\sin\big((\frac{1}{2}-\beta)\xi\big)\sin\frac{3}{2}\xi\big]
\nonumber\\
&&-(1-\frac{\beta}{2})\big[\cos\big((\frac{1}{2}-\beta)\xi\big)\cos\frac{1}{2}\xi
+\sin\big((\frac{1}{2}-\beta)\xi\big)\sin\frac{1}{2}\xi\big]
\nonumber\\
&=&-\cos\big((\frac{1}{2}-\beta)\xi\big)\big[1+\frac{\beta}{2}\big(\cos\frac{3}{2}\xi-\cos\frac{1}{2}\xi\big)\big]
\nonumber\\
&&-\sin\big((\frac{1}{2}-\beta)\xi\big)\big[1-\frac{\beta}{2}\big(\sin\frac{3}{2}\xi+\sin\frac{1}{2}\xi\big)\big]
\nonumber\\
&\leq& 0.
\end{eqnarray}
\item
for the scheme corresponding to (\ref{equation2.114}) with $(p,q)=(0,1)$, where $d_{-1}=0,~d_0=\beta,~d_1=1-\beta$,
\begin{equation}\label{equation3.25*2}
g(\beta)=\beta\cos\big((1-\beta)\xi\big)-(1-\beta)\cos\beta\xi,
\end{equation}
and because of

\begin{eqnarray}\label{equation3.26}
&&g'(\beta)
\nonumber\\
&=&\cos\big((1-\beta)\xi\big)+\beta\xi\sin\big((1-\beta)\xi\big)
+\cos\beta\xi+(1-\beta)\xi\sin\beta\xi
\nonumber\\
&=&\cos\xi\cos\beta\xi+\sin\xi\sin\beta\xi
\nonumber\\
&&+\beta\xi\sin\big((1-\beta)\xi\big)
+\cos\beta\xi+(1-\beta)\xi\sin\beta\xi
\nonumber\\
&=&\big(1+\cos\xi\big)\cos\beta\xi+\sin\xi\sin\beta\xi
\nonumber\\
&&+\beta\xi\sin\big((1-\beta)\xi\big)+(1-\beta)\xi\sin\beta\xi
\geq 0,
\end{eqnarray}
$g(\beta)$ increases with respect to $\beta$. Therefore, $g(\beta)\leq g(\frac{1}{2})=0$.

\item
for the scheme corresponding to (\ref{equation2.311}), after ignoring the common factor $\frac{\beta(1-\beta)}{12}$, $d_{-1}=0,~d_0=(6\beta-1),~d_1=(13-6\beta)$,
\begin{eqnarray}\label{equation3.27}
&&g(\beta)
\nonumber\\
&=&(6\beta-1)\cos\big((1-\beta)\xi\big)-(13-6\beta)\cos\beta\xi
\nonumber\\
&=&6\beta\cos\big((1-\beta)\xi\big)-6(1-\beta)\cos\beta\xi
-\cos\big((1-\beta)\xi\big)-7\cos\beta\xi
\nonumber\\
&=&6\beta g_1(\beta)-g_2(\beta),
\end{eqnarray}
where
\begin{equation}\label{equation3.28}
g_1(\beta)=\beta\cos\big((1-\beta)\xi\big)-(1-\beta)\cos\beta\xi,
\end{equation}
\begin{eqnarray}\label{equation3.29}
&&g_2(\beta)
\nonumber\\
&=&\cos\big((1-\beta)\xi\big)+7\cos\beta\xi
\nonumber\\
&=&(7+\cos\xi)\cos\beta\xi+\sin\xi\sin\beta\xi>0,
\end{eqnarray}
from (\ref{equation3.23}), (\ref{equation3.25*2}) and (\ref{equation3.26}), there exists $g(\beta)< 0$.

\item
for the scheme corresponding to (\ref{equation2.312}), after ignoring the common factor $\frac{\beta(1-\beta)(6\beta-1)}{12}$,
$d_{-1}=0,~d_0=1,~d_1=-1$,
\begin{eqnarray}\label{equation3.210*1}
&&g(\beta)
\nonumber\\
&=&\cos\big((1-\beta)\xi\big)+\cos\beta\xi
\nonumber\\
&=&\cos\xi\cos\beta\xi+\sin\xi\sin\beta\xi+\cos\beta\xi
\nonumber\\
&=&\cos\beta\xi(1+\cos\xi)+\sin\xi\sin\beta\xi>0.
\end{eqnarray}
\end{enumerate}

From Lemmas \ref{lemma3.22} and \ref{lemma3.23}, we know that $Re(\lambda)<0$ for the second order schemes discussed in this proof and the third quasi-compact schemes corresponding to (\ref{equation2.38})-(\ref{equation2.313}). And the desired result follows from Lemma \ref{lemma3.21}. The proof is completed.
\end{proof}

%

\subsection{Stability and convergent analysis}\label{subsec:3.3}

As for the stability of the scheme (\ref{equation3.18}), we have the following results.
\begin{theorem}\label{the3.31}
The CN-quasi-compact scheme (\ref{equation3.18}) is unconditionally stable if $\textbf{T}$ is symmetric and positive definite and
$\textbf{A}$ is negative definite.
\end{theorem}
\begin{proof}
Denoting $\widetilde{\textbf{A}}:=\frac{\tau}{2h^{\alpha}}(K_1 \textbf{A}+K_2 \textbf{A}^{T})$, we prove that the magnitudes of the eigenvalues of the iterative matrix $(\textbf{T}-\widetilde{\textbf{A}})^{-1}(\textbf{T}+\widetilde{\textbf{A}})$ are less than one.

Firstly, we show that the real parts of all the eigenvalues of $\textbf{C}:=(\textbf{T})^{-1}\widetilde{\textbf{A}}$ are negative. In fact, if $\lambda$ is an eigenvalue of $\textbf{C}$, and $\textbf{x}$ is the
corresponding eigenvector, then $(\textbf{T})^{-1}\widetilde{\textbf{A}}\textbf{x}=\lambda\textbf{x}
~\Rightarrow~ \widetilde{\textbf{A}}\textbf{x}=\lambda\textbf{T}\textbf{x} ~\Rightarrow~
\textbf{x}^T\widetilde{\textbf{A}}\textbf{x}=\lambda\textbf{x}^T\textbf{T}\textbf{x} ~\Rightarrow~
\lambda=\textbf{x}^T\widetilde{\textbf{A}}\textbf{x}/\textbf{x}^T\textbf{T}\textbf{x}$.
Since $\textbf{T}$ is symmetric and positive definite, $\textbf{x}^T\textbf{T}\textbf{x}>0$; and because of
$\widetilde{\textbf{A}}$ being negative definite, $Re(\textbf{x}^T\widetilde{\textbf{A}}\textbf{x})<0$. Therefore, $Re(\lambda)<0$.

Secondly, since $\textbf{B}:=(\textbf{T}-\widetilde{\textbf{A}})^{-1}(\textbf{T}+\widetilde{\textbf{A}})
=(\textbf{I}-\textbf{C})^{-1}(\textbf{I}+\textbf{C})$, if $\lambda$ is an eigenvalue of $\textbf{C}$, then
$(1+\lambda)/(1-\lambda)$ is an eigenvalue of $\textbf{B}$. Therefore $|(1+\lambda)/(1-\lambda)|<1$, i.e.,
the CN-quasi-compact scheme is unconditionally stable. The proof is completed.
\end{proof}

The convergence results are listed in the following lemma and theorem. Since the error estimates here in Theorem \ref{the3.32}
can be obtained by Lemma \ref{lemma3.32} and similar analysis as in \cite{Zhou:13}, the proof is omitted.

\begin{lemma}\cite{Yu:04}\label{lemma3.32}
The eigenvalues of $\textbf{T}$ in (\ref{equation3.19}) are given as
\begin{equation}
\lambda_{j}=c_0+2\sqrt{c_{-1}c_{1}}\cos(j\pi/N),~~~1\leq j \leq N-1.
\end{equation}
Moreover, if $\textbf{T}$ is symmetric and strictly positive definite, its eigenvalues are real and satisfy
\begin{equation}
\lambda_{j}\geq c_0-2|c_{1}|>0~~~{\rm or}~~~\lambda_{j}\geq c_0-2|c_{-1}|>0,~~~1\leq j \leq N-1.
\end{equation}
\end{lemma}

\begin{theorem}\label{the3.32}
Let $u_{i}^n$ be the exact solution of problem (\ref{equation1.1}), and $U_{i}^n$ be the solution of the $m$-th order CN-quasi-compact scheme (\ref{equation3.18}) at grid point $(x_i,t_n)$. If $\textbf{T}$ is symmetric and strictly positive definite and $\textbf{A}$ is negative definite, the estimate
\begin{equation}
\|u^n-U^n\|\leq C(\tau^2+h^m),~~~1\leq n \leq M,
\end{equation}
holds, where $\|\cdot\|$ means the discrete $L^2$ norm.
\end{theorem}

\begin{remark}\label{remark3.31}
When $\textbf{T}$ is asymmetric, we find numerically that the CN-quasi-compact scheme (\ref{equation3.18}) is also stable and convergent, if
$\textbf{T}$ is strictly positive definite and $\textbf{A}$ is negative definite, or in more relaxed as well as easier to judge conditions:
$c_{-1}+c_0+c_1=d_{-1}+d_0+d_1$, $c_{0}>|c_{1}|+|c_{-1}|$, $d_{1}>|d_{0}|+|d_{-1}|$, and (\ref{equation2.37}) holds. Although we cannot theoretically prove it, it is believed to be true.
\end{remark}
From Theorem \ref{the3.31}, the second order CN-quasi-compact schemes deduced from (\ref{equation2.114}) are unconditionally stable
if $c_{-1}=c_1$ and $\textbf{A}$ is negative definite. All the second order quasi-compact approximations that can derive the CN-quasi-compact unconditionally stable schemes are listed in Table (\ref{table3.1}); and they are numbered from $1$ to $10$; in particular, the ones from No. 4 to No. 10 can be directly used to solve the two-sided problem (\ref{equation1.1}) with $K_1K_2\neq 0$, since $c_{-1}=c_1$.


\begin{table}
\caption{Second order quasi-compact approximations that can derive the CN-quasi-compact unconditionally stable schemes}\label{table3.1}
\begin{tabular}{ccc|c|cccc|ccc}
\hline\noalign{\smallskip}
 \multicolumn{3}{c|}{} & \multicolumn{8}{|c}{Eq. (\ref{equation2.114})}  \\
\cline{4-11}
 Eq. & Eq. & Eq. & $(p,q)$ &\multicolumn{4}{|c|}{$(0,1)$} &  \multicolumn{3}{|c}{$(-1,1)$}\\
\cline{4-11}
(\ref{equation2.120})   & (\ref{equation2.117})                & (\ref{equation2.116})
                                     & $(m_1,n_1)$ & $(1,2)$&  $(1,2)$& $(1,3)$& $(2,3)$& $(2,3)$& $(2,4)$& $(3,4)$\\
\multicolumn{3}{c|}{}& $(m_2,n_2)$ & $(0,1)$&  $(1,2)$& $(0,1)$& $(0,1)$& $(0,1)$& $(0,1)$& $(0,1)$\\
\noalign{\smallskip}\hline\noalign{\smallskip}
1&2&3&&4&5&6&7&8&9&10\\
\noalign{\smallskip}\hline
\end{tabular}
\end{table}

Similarly, for the third order schemes, we can still analyze the properties of $\textbf{A}$, $\textbf{T}$, $\textbf{D}^{-1}\textbf{T}$, to prove their unconditional stability. For the simplicity, we list the stable cases in Table \ref{table3.2}; some of them can be easily proven in theory, while others are numerically verified.
\begin{table}
\caption{Third order quasi-compact approximations that can derive the CN-quasi-compact unconditionally stable schemes theoretically proved or numerically verified}\label{table3.2}
\begin{tabular}{c|ccccc}
\hline\noalign{\smallskip}
                 &$(1,2)$ & -      &-       &-       &-        \\
 here 1,$\cdots$, 10             &$(1,3)$ &        &-       & -      &-         \\
 correspond to           &$(1,4)$ &$(2,4)$ & -      &  -     &-         \\
the second order        &$(1,5)$ &$(2,5)$ &$(3,5)$ &$(4,5)$ &-        \\
approximations        &$(1,6)$ &$(2,6)$ & -      &   -    & -         \\
No. 1,$\cdots$, No. 10          &$(1,7)$ &$(2,7)$ &$(3,7)$ &     -  & -      \\
 given in Table \ref{table3.1};    &$(1,8)$ &$(2,8)$ &$(3,8)$ &   -    &$(5,8)$ \\
and here $(\cdot,\cdot)$ means the                    &$(1,9)$ &$(2,9)$ &$(3,9)$ &    -   &$(5,9)$ \\
 combination of two approximations                &$(1,10)$&$(2,10)$&$(3,10)$& -      &$(5,10)$\\
\noalign{\smallskip}\hline
\end{tabular}
\end{table}

For the fourth order quasi-compact scheme, we numerically show that the scheme obtained by combining $(1,2)$ and $(1,4)$ is unstable; see the eigenvalues distribution of the iterative matrix in Figure \ref{fig3.1} with $\alpha=1.5,\,N=100$. While combining $(1,2)$ and $(1,8)$ leads to an unconditionally stable fourth order quasi-compact scheme; Figure \ref{fig3.2} shows the eigenvalues distribution of the iterative matrix  with $\alpha=1.5,\,N=100$.

\begin{figure}[!htbp]
\includegraphics[scale=0.4]{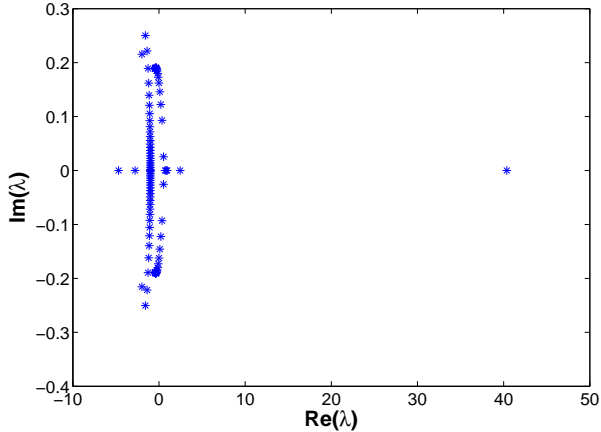}\\
\caption{Eigenvalues of the scheme that is designed by $(1,2)$ and $(1,4)$, where $\alpha=1.5,~N=100$.}\label{fig3.1}
\end{figure}

\begin{figure}[!htbp]
\includegraphics[scale=0.4]{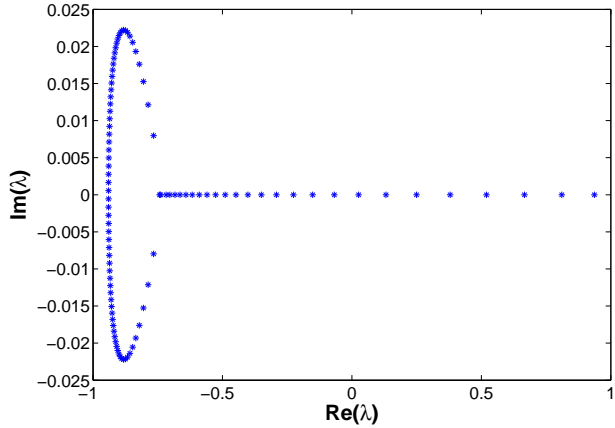}\\
\caption{Eigenvalues of the scheme that is designed by $(1,2)$ and $(1,8)$, where $\alpha=1.5,~N=100$.}\label{fig3.2}
\end{figure}

\begin{remark}\label{remark3.32}
While, usually, for a numerical approximation, we only need the regularity for the function $u(x)$ itself.
However, for some of the approximations to fractional derivatives it seems strange that more conditions are actually needed even for the shifted Gr\"{u}nwald-Letnikov formula with first order of accuracy ($u\in L^1(\mathbb{R})$ and $u\in C^{\alpha+1}(\mathbb{R})$ \cite{Meerschaert:04}); this is because the schemes are derived by the Fourier analysis which uses information of the whole domain. Although Theorem \ref{the2.21} in this paper shows that the conditions, which are less rigorous and easier to judge, are enough for general high order schemes; it still means more boundary conditions are needed for the space fractional diffusion problem (\ref{equation1.1}).

By using the techniques similar to \cite{Chen:14}, the schemes discussed in this paper can still be effective and keep the corresponding high order accuracy after removing the extra boundary requirements. More details for the approximations to the Riemann-Liouville fractional derivatives are described in Appendix.  A numerical example is given in the next section to demonstrate that these techniques also work for the time-dependent PDEs.

\end{remark}

\section{Numerical experiments}\label{sec:4}
We perform numerical experiments to show the powerfulness of the derived schemes and confirm the theoretical analysis and convergence orders.
\begin{example}\label{example4.1}
Consider the following problem
\begin{equation}\label{equation4.1}
\frac{\partial u(x,t)}{\partial t}=
\,_{a}D_{x}^{\alpha}u(x,t)-e^{-t}\big(x^{3+\alpha}+\frac{\Gamma(4+\alpha)}{\Gamma(4)}x^3\big),
~~~(x,t)\in(0,1)\times(0,1],
\end{equation}
with the boundary conditions
\begin{equation}
u(0,t)=0,~~~u(1,t)=e^{-t},~~~t\in[0,1],
\end{equation}
and the initial value
\begin{equation}
u(x,0)=x^{3+\alpha},~~~x\in[0,1].
\end{equation}
Then the exact solutions of (\ref{equation4.1}) is $e^{-t}x^{3+\alpha}$.
\end{example}
Letting $\tau=h$, $\tau=h/20$, and $\tau=h^2$, respectively, in second, third and fourth order (in terms of spatial direction) stable schemes
can make sure that the numerical errors caused by the Crank-Nicolson method in time direction is small enough, so that errors in spatial direction are dominant and the convergence rates can be testified.
The numerical results in Tables \ref{table4.1} and \ref{table4.2} confirm the convergence orders of the corresponding  CN-quasi-compact schemes.

\begin{table}
\caption{The discrete $L^2$ errors and their convergence rates to Example \ref{example4.1} at $t=1$
by using the second order stable CN-quasi-compact schemes and the fourth order one $(1,2)$+$(1,8)$
for different $\alpha$ with $\tau=h$ and $\tau=h^2$, respectively.}\label{table4.1}
\begin{tabular}{cccccccccc}
\hline\noalign{\smallskip}
\multicolumn{2}{c}{Number}  &\multicolumn{2}{c}{$1$} &\multicolumn{2}{c}{$2$} &\multicolumn{2}{c}{$3$}
       &\multicolumn{2}{c}{$(1,2)$+$(1,8)$}\\
\noalign{\smallskip}\hline\noalign{\smallskip}
$\alpha$&$N$& $\|u-U\|$& rate & $\|u-U\|$ & rate & $\|u-U\|$ & rate & $\|u-U\|$ & rate\\
\noalign{\smallskip}\hline\noalign{\smallskip}
1.1    & 8 & 2.97 1e-3 & -    & 4.07 1e-3 & -    & 2.24 1e-2 & -    & 1.02 1e-5 & -  \\
       & 16& 8.49 1e-4 & 1.81 & 1.06 1e-3 & 1.95 & 5.54 1e-3 & 1.92 & 7.29 1e-7 & 3.81 \\
       & 32& 2.27 1e-4 & 1.90 & 2.68 1e-4 & 1.98 & 1.53 1e-3 & 1.95 & 4.83 1e-8 & 3.92 \\
       & 64& 5.86 1e-5 & 1.95 & 6.75 1e-5 & 1.99 & 3.91 1e-4 & 1.97 & 3.10 1e-9 & 3.96 \\
       &128& 1.49 1e-5 & 1.98 & 1.69 1e-5 & 1.99 & 9.87 1e-4 & 1.99 & 1.97 1e-10& 3.98 \\
\noalign{\smallskip}\hline\noalign{\smallskip}
1.5    & 8 & 2.87 1e-3 & -    & 1.08 1e-3 & -    & 9.49 1e-3 & -    & 5.75 1e-6 & -\\
       & 16& 7.78 1e-4 & 1.88 & 2.57 1e-4 & 2.07 & 2.50 1e-3 & 1.92 & 4.11 1e-7 & 3.80\\
       & 32& 2.02 1e-4 & 1.94 & 6.26 1e-5 & 2.04 & 6.44 1e-4 & 1.96 & 2.74 1e-8 & 3.91\\
       & 64& 5.16 1e-5 & 1.97 & 1.54 1e-5 & 2.02 & 1.63 1e-4 & 1.98 & 1.77 1e-9 & 3.95\\
       &128& 1.30 1e-5 & 1.99 & 3.84 1e-6 & 2.01 & 4.11 1e-5 & 1.99 & 1.12 1e-10& 3.98\\
\noalign{\smallskip}\hline\noalign{\smallskip}
1.9    & 8 & 3.13 1e-3 & -    & 2.28 1e-3 & -    & 4.08 1e-3 & -    & 4.52 1e-6 & -\\
       & 16& 7.97 1e-4 & 1.98 & 5.80 1e-4 & 1.97 & 1.04 1e-3 & 1.98 & 2.99 1e-7 & 3.92 \\
       & 32& 2.01 1e-4 & 1.99 & 1.46 1e-4 & 1.99 & 2.61 1e-4 & 1.99 & 1.92 1e-8 & 3.96\\
       & 64& 5.04 1e-5 & 1.99 & 3.68 1e-5 & 1.99 & 6.55 1e-5 & 1.99 & 1.22 1e-9 & 3.98\\
       &128& 1.26 1e-5 & 2.00 & 9.21 1e-6 & 2.00 & 1.64 1e-5 & 2.00 & 7.66 1e-11& 3.99\\
\noalign{\smallskip}\hline
\end{tabular}
\end{table}

\begin{table}
\caption{The discrete $L^2$ errors and their convergence rates to Example \ref{example4.1} at $t=1$
by using the third order stable CN-quasi-compact schemes for different $\alpha$ with $\tau=h/20$.}\label{table4.2}
\begin{tabular}{cccccccccc}
\hline\noalign{\smallskip}
\multicolumn{2}{c}{Number}  &\multicolumn{2}{c}{$(1,3)$} &\multicolumn{2}{c}{$(1,4)$} &\multicolumn{2}{c}{$(1,5)$}
       &\multicolumn{2}{c}{$(1,8)$}\\
\noalign{\smallskip}\hline\noalign{\smallskip}
$\alpha$&$N$& $\|u-U\|$& rate & $\|u-U\|$ & rate & $\|u-U\|$ & rate & $\|u-U\|$ & rate\\
\noalign{\smallskip}\hline\noalign{\smallskip}
1.1    & 8 & 4.60 1e-4 & -    & 5.72 1e-4 & -    & 3.31 1e-4 & -    & 3.15 1e-4 & -  \\
       & 16& 6.02 1e-5 & 2.93 & 7.78 1e-5 & 2.88 & 4.15 1e-5 & 2.99 & 4.17 1e-5 & 2.92 \\
       & 32& 7.65 1e-6 & 2.98 & 1.01 1e-5 & 2.95 & 5.16 1e-6 & 3.01 & 5.31 1e-6 & 2.97 \\
       & 64& 9.53 1e-7 & 3.01 & 1.27 1e-6 & 2.99 & 6.32 1e-7 & 3.03 & 6.58 1e-7 & 3.01 \\
       &128& 1.16 1e-7 & 3.04 & 1.56 1e-7 & 3.02 & 7.52 1e-8 & 3.07 & 7.90 1e-8 & 3.06 \\
\noalign{\smallskip}\hline\noalign{\smallskip}
1.5    & 8 & 2.22 1e-4 & -    & 2.55 1e-4 & -    & 1.83 1e-4 & -    & 9.42 1e-6 & -  \\
       & 16& 2.85 1e-5 & 2.96 & 3.39 1e-5 & 2.91 & 2.27 1e-5 & 3.01 & 1.57 1e-6 & 2.58 \\
       & 32& 3.58 1e-6 & 2.99 & 4.34 1e-6 & 2.97 & 2.81 1e-6 & 3.02 & 2.44 1e-7 & 2.69 \\
       & 64& 4.43 1e-7 & 3.02 & 5.42 1e-7 & 3.00 & 3.43 1e-7 & 3.04 & 4.01 1e-8 & 2.60 \\
       &128& 5.33 1e-8 & 3.05 & 6.60 1e-8 & 3.04 & 4.06 1e-8 & 3.08 & 7.36 1e-9 & 2.45 \\
\noalign{\smallskip}\hline\noalign{\smallskip}
1.9    & 8 & 3.75 1e-5 & -    & 4.05 1e-5 & -    & 1.26 1e-4 & -    & 3.25 1e-4 & -  \\
       & 16& 4.63 1e-6 & 3.02 & 5.43 1e-6 & 2.90 & 1.53 1e-5 & 3.04 & 4.23 1e-5 & 2.94\\
       & 32& 5.57 1e-7 & 3.06 & 7.18 1e-7 & 2.92 & 1.87 1e-6 & 3.03 & 5.39 1e-6 & 2.97 \\
       & 64& 6.37 1e-8 & 3.13 & 9.97 1e-8 & 2.89 & 2.26 1e-7 & 3.05 & 6.84 1e-7 & 2.98 \\
       &128& 6.48 1e-9 & 3.30 & 1.37 1e-8 & 2.82 & 2.67 1e-8 & 3.09 & 8.73 1e-8 & 2.97 \\
\noalign{\smallskip}\hline
\end{tabular}
\end{table}

\begin{example}\label{example4.2}
In the domain $(0,1)\times (0,1)$, consider the following problem
\begin{equation}\label{equation4.2}
\left\{ \begin{array}{lll}
\frac{\partial u(x,t) }{\partial t}&=&\,_{0}D_x^{\alpha}u(x,t)+\, _{x}D_{1}^{\alpha}u(x,t)+f(x,t), \\
u(x,0) &=&x^3(1-x)^3 ~~~~ {\rm for}~~~ x\in[0,1], \\
u(0,t)&=&u(1,t)=0 ~~~~{\rm for}~~~ t\in[0,1],\\
\end{array} \right.
\end{equation}
with the source term
\begin{eqnarray*}
&&f(x,t)
\nonumber\\
&=&-e^{-t}\bigg(x^3(1-x)^3+\frac{\Gamma(4)}{\Gamma(4-\alpha)}\big(x^{3-\alpha}+(1-x)^{3-\alpha}\big)
\nonumber\\
&&~~~~~~~~~-3\times\frac{\Gamma(5)}{\Gamma(5-\alpha)}\big(x^{4-\alpha}+(1-x)^{4-\alpha}\big)
\nonumber\\
&&~~~~~~~~~+3\times\frac{\Gamma(6)}{\Gamma(6-\alpha)}\big(x^{5-\alpha}+(1-x)^{5-\alpha}\big)
\nonumber\\
&&~~~~~~~~~-\frac{\Gamma(7)}{\Gamma(7-\alpha)}\big(x^{6-\alpha}+(1-x)^{6-\alpha}\big)
\bigg).
\end{eqnarray*}
Then the exact solutions of (\ref{equation4.2}) is $e^{-t}x^{3}(1-x)^3$.
\end{example}
Letting $\tau=h$ and $\tau=h/20$ respectively in second and third order (in terms of spatial direction) stable schemes
can make sure that the numerical errors caused by the Crank-Nicolson method in time direction is small enough, so that errors in spatial direction are dominant and the convergence rates can be testified. The numerical results in Tables \ref{table4.3}, \ref{table4.4} confirm the theoretical convergence orders $2$ and $3$.

\begin{table}
\caption{The discrete $L^2$ errors and their convergence rates to Example \ref{example4.2} at $t=1$
by using the second order stable CN-quasi-compact schemes
for different $\alpha$ with $\tau=h$.}\label{table4.3}
\begin{tabular}{cccccccccc}
\hline\noalign{\smallskip}
\multicolumn{2}{c}{Number}   &\multicolumn{2}{c}{$4$} &\multicolumn{2}{c}{$5$} &\multicolumn{2}{c}{$8$}
       &\multicolumn{2}{c}{$9$}\\
\noalign{\smallskip}\hline\noalign{\smallskip}
$\alpha$&$N$& $\|u-U\|$& rate & $\|u-U\|$ & rate & $\|u-U\|$ & rate & $\|u-U\|$ & rate\\
\noalign{\smallskip}\hline\noalign{\smallskip}
1.1    & 8 & 2.67 1e-4 & -    & 1.86 1e-4 & -    & 3.12 1e-4 & -    & 2.61 1e-4 & -  \\
       & 16& 6.83 1e-5 & 1.97 & 4.12 1e-5 & 2.17 & 5.53 1e-5 & 2.49 & 5.88 1e-5 & 2.15 \\
       & 32& 1.75 1e-5 & 1.97 & 9.51 1e-6 & 2.12 & 2.00 1e-5 & 1.47 & 2.57 1e-5 & 1.19 \\
       & 64& 4.44 1e-6 & 1.97 & 2.27 1e-6 & 2.06 & 7.40 1e-6 & 1.44 & 9.03 1e-6 & 1.51 \\
       &128& 1.12 1e-7 & 1.98 & 5.55 1e-7 & 2.03 & 2.23 1e-6 & 1.73 & 2.65 1e-6 & 1.77 \\
\noalign{\smallskip}\hline\noalign{\smallskip}
1.5    & 8 & 2.18 1e-4 & -    & 7.25 1e-5 & -    & 2.78 1e-4 & -    & 3.46 1e-4 & -\\
       & 16& 5.35 1e-5 & 2.03 & 1.54 1e-5 & 2.24 & 8.43 1e-5 & 1.72 & 1.01 1e-4 & 1.77\\
       & 32& 1.35 1e-5 & 1.99 & 3.44 1e-6 & 2.16 & 2.78 1e-5 & 1.89 & 2.70 1e-5 & 1.91\\
       & 64& 3.39 1e-6 & 1.99 & 8.14 1e-7 & 2.08 & 5.94 1e-6 & 1.94 & 6.99 1e-6 & 1.95\\
       &128& 8.52 1e-7 & 1.99 & 1.98 1e-7 & 2.04 & 1.52 1e-6 & 1.97 & 1.78 1e-6 & 1.97\\
\noalign{\smallskip}\hline\noalign{\smallskip}
1.9    & 8 & 1.14 1e-4 & -    & 6.48 1e-5 & -    & 1.38 1e-4 & -    & 1.51 1e-4 & -\\
       & 16& 3.50 1e-5 & 1.70 & 1.84 1e-5 & 1.81 & 4.25 1e-5 & 1.70 & 4.66 1e-5 & 1.70 \\
       & 32& 8.72 1e-6 & 2.00 & 4.58 1e-6 & 2.01 & 1.08 1e-5 & 1.98 & 1.18 1e-5 & 1.98\\
       & 64& 2.18 1e-6 & 2.00 & 1.14 1e-6 & 2.00 & 2.71 1e-6 & 1.99 & 2.97 1e-6 & 1.99\\
       &128& 5.44 1e-7 & 2.00 & 2.86 1e-7 & 2.00 & 6.79 1e-7 & 2.00 & 7.44 1e-7 & 2.00\\
\noalign{\smallskip}\hline
\end{tabular}
\end{table}

\begin{table}
\caption{The discrete $L^2$ errors and their convergence rates to Example \ref{example4.2} at $t=1$
by using the third order stable CN-quasi-compact schemes for different $\alpha$ with $\tau=h/20$.}\label{table4.4}
\begin{tabular}{cccccccccc}
\hline\noalign{\smallskip}
\multicolumn{2}{c}{Number}   &\multicolumn{2}{c}{$(4,5)$} &\multicolumn{2}{c}{$(5,8)$} &\multicolumn{2}{c}{$(5,9)$}
       &\multicolumn{2}{c}{$(5,10)$}\\
\noalign{\smallskip}\hline\noalign{\smallskip}
$\alpha$&$N$& $\|u-U\|$& rate & $\|u-U\|$ & rate & $\|u-U\|$ & rate & $\|u-U\|$ & rate\\
\noalign{\smallskip}\hline\noalign{\smallskip}
1.1    & 8 & 4.64 1e-5 & -    & 2.20 1e-4 & -    & 2.04 1e-4 & -    & 1.56 1e-4 & -  \\
       & 16& 6.68 1e-6 & 2.80 & 3.80 1e-5 & 2.54 & 3.46 1e-5 & 2.56 & 2.53 1e-5 & 2.63\\
       & 32& 9.12 1e-7 & 2.86 & 5.47 1e-6 & 2.80 & 4.95 1e-6 & 2.80 & 3.57 1e-6 & 2.83 \\
       & 64& 1.22 1e-7 & 2.91 & 7.36 1e-7 & 2.89 & 6.66 1e-7 & 2.89 & 4.78 1e-7 & 2.90 \\
       &128& 1.60 1e-8 & 2.94 & 9.62 1e-8 & 2.94 & 8.69 1e-8 & 2.94 & 6.22 1e-8 & 2.94 \\
\noalign{\smallskip}\hline\noalign{\smallskip}
1.5    & 8 & 3.36 1e-5 & -    & 5.12 1e-5 & -    & 4.89 1e-5 & -    & 4.39 1e-5 & -    \\
       & 16& 4.66 1e-6 & 2.85 & 7.60 1e-6 & 2.75 & 7.22 1e-6 & 2.76 & 6.39 1e-6 & 2.78 \\
       & 32& 6.51 1e-7 & 2.84 & 1.08 1e-6 & 2.82 & 1.02 1e-6 & 2.82 & 9.03 1e-7 & 2.82 \\
       & 64& 8.97 1e-8 & 2.86 & 1.49 1e-7 & 2.86 & 1.41 1e-7 & 2.86 & 1.25 1e-7 & 2.86 \\
       &128& 1.21 1e-8 & 2.88 & 2.01 1e-8 & 2.89 & 1.91 1e-8 & 2.89 & 1.68 1e-8 & 2.89 \\
\noalign{\smallskip}\hline\noalign{\smallskip}
1.8    & 8 & 2.82 1e-5 & -    & 2.15 1e-5 & -    & 2.21 1e-5 & -    & 2.35 1e-5 & -  \\
       & 16& 2.90 1e-6 & 3.29 & 1.16 1e-6 & 4.22 & 1.33 1e-6 & 4.05 & 1.77 1e-6 & 3.74\\
       & 32& 3.46 1e-7 & 3.07 & 5.40 1e-8 & 4.42 & 8.67 1e-8 & 3.94 & 1.64 1e-7 & 3.43 \\
       & 64& 4.43 1e-8 & 2.97 & 1.98 1e-9 & 4.77 & 7.33 1e-9 & 3.58 & 1.86 1e-8 & 3.14 \\
       &128& 5.81 1e-9 & 2.93 & 1.67 1e-10& 3.57 & 8.31 1e-10& 3.12 & 2.36 1e-9 & 2.98 \\
\noalign{\smallskip}\hline
\end{tabular}
\end{table}

\begin{example}\label{example4.add}
Consider the nonhomogeneous problem
\begin{equation}\label{equation4.add1}
\left\{ \begin{array}{lll}
\frac{\partial u(x,t) }{\partial t}&=&\,_{0}D_x^{\alpha}u(x,t)+f(x,t), \\
u(x,0) &=&1+x+x^{3+\alpha} ~~~~ {\rm for}~~~ x\in[0,1], \\
u(0,t)&=&e^{-t} ~~~~{\rm for}~~~ t\in[0,1],\\
u(1,t)&=&3e^{-t} ~~~~{\rm for}~~~ t\in[0,1],\\
\end{array} \right.
\end{equation}
with the source term
\begin{eqnarray*}
&&f(x,t)
\nonumber\\
&=&-e^{-t}\bigg(1+x+x^{3+\alpha}+\frac{1}{\Gamma(1-\alpha)}x^{-\alpha}
\nonumber\\
&&~~~~~~~~~+\frac{1}{\Gamma(2-\alpha)}x^{1-\alpha}+\frac{\Gamma(4+\alpha)}{\Gamma(4)}x^{3}
\bigg).
\end{eqnarray*}
Then the exact solutions of (\ref{equation4.add1}) is $e^{-t}(1+x+x^{3+\alpha})$.
\end{example}
We apply the CN-quasi-compact schemes to (\ref{equation4.add1}).
Letting $\tau=h$, $\tau=h/20$, and $\tau=h^2$, respectively in second, third, and fourth order (in terms of spatial direction) stable schemes
can make sure that the numerical errors caused by the Crank-Nicolson method in time direction is small enough, so that errors in spatial direction are dominant and the convergence rates can be testified. The numerical results in Tables \ref{table4.add1} and \ref{table4.add2} indicate the effectiveness of these schemes for nonhomogeneous problems.

\begin{table}
\caption{The discrete $L^2$ errors and their convergence rates to Example \ref{example4.add} at $t=1$
by using the second order stable CN-pseudo-compact schemes and the fourth order one $(1,2)$+$(1,8)$
for different $\alpha$ with $\tau=h$ and $\tau=h^2$, respectively.}\label{table4.add1}
\begin{tabular}{cccccccccc}
\hline\noalign{\smallskip}
\multicolumn{2}{c}{Number}   &\multicolumn{2}{c}{$3$} &\multicolumn{2}{c}{$4$} &\multicolumn{2}{c}{$8$} &\multicolumn{2}{c}{$(1,2)+(1,8)$}\\
\noalign{\smallskip}\hline\noalign{\smallskip}
$\alpha$&$N$& $\|u-U\|$& rate & $\|u-U\|$ & rate & $\|u-U\|$ & rate & $\|u-U\|$ & rate\\
\noalign{\smallskip}\hline\noalign{\smallskip}
1.1    & 8 & 2.48 1e-3 & -    & 1.31 1e-3 & -    & 3.56 1e-3 & -    & 1.75 1e-5 & -  \\
       & 16& 8.50 1e-4 & 1.53 & 3.83 1e-4 & 1.78 & 8.67 1e-4 & 2.04 & 1.22 1e-6 & 3.84\\
       & 32& 2.58 1e-4 & 1.72 & 1.22 1e-4 & 1.65 & 2.64 1e-4 & 1.72 & 8.06 1e-7 & 3.92 \\
       & 64& 6.68 1e-5 & 1.95 & 3.18 1e-5 & 1.95 & 7.01 1e-5 & 1.91 & 5.17 1e-9 & 3.96\\
       &128& 1.78 1e-5 & 1.99 & 7.97 1e-6 & 1.99 & 1.77 1e-5 & 1.99 & 3.33 1e-10& 3.96 \\
\noalign{\smallskip}\hline\noalign{\smallskip}
1.5    & 8 & 6.28 1e-3 & -    & 5.30 1e-3 & -    & 7.76 1e-3 & -    & 5.66 1e-6 & -\\
       & 16& 1.58 1e-3 & 1.99 & 1.29 1e-4 & 2.04 & 1.97 1e-3 & 1.98 & 4.75 1e-7 & 3.57\\
       & 32& 3.93 1e-4 & 2.01 & 3.12 1e-4 & 2.05 & 4.91 1e-4 & 2.00 & 3.48 1e-8 & 3.77\\
       & 64& 9.64 1e-4 & 2.03 & 7.51 1e-5 & 2.05 & 1.21 1e-4 & 2.02 & 2.35 1e-9 & 3.88\\
       &128& 2.36 1e-5 & 2.03 & 1.81 1e-5 & 2.05 & 2.97 1e-5 & 2.03 & 1.47 1e-10& 4.00\\
\noalign{\smallskip}\hline\noalign{\smallskip}
1.9    & 8 & 4.11 1e-3 & -    & 4.04 1e-3 & -    & 4.74 1e-3 & -    & 3.41 1e-6 & -\\
       & 16& 1.04 1e-3 & 1.98 & 1.02 1e-3 & 1.98 & 1.22 1e-3 & 1.96 & 2.70 1e-7 & 3.66\\
       & 32& 2.62 1e-4 & 1.99 & 2.57 1e-4 & 1.99 & 3.01 1e-4 & 1.98 & 1.92 1e-8 & 3.81\\
       & 64& 6.58 1e-5 & 1.99 & 6.44 1e-5 & 2.00 & 7.82 1e-5 & 1.99 & 1.28 1e-9 & 3.90\\
       &128& 1.65 1e-5 & 2.00 & 1.62 1e-5 & 2.00 & 1.96 1e-5 & 1.99 & 8.41 1e-11& 3.93\\
\noalign{\smallskip}\hline
\end{tabular}
\end{table}

\begin{table}
\caption{The discrete $L^2$ errors and their convergence rates to Example \ref{example4.add} at $t=1$
by using the third order stable CN-pseudo-compact schemes for different $\alpha$ with $\tau=h/20$.}\label{table4.add2}
\begin{tabular}{cccccccccc}
\hline\noalign{\smallskip}
\multicolumn{2}{c}{Number}  &\multicolumn{2}{c}{$(1,3)$} &\multicolumn{2}{c}{$(1,4)$} &\multicolumn{2}{c}{$(1,5)$}
       &\multicolumn{2}{c}{$(4,5)$}\\
\noalign{\smallskip}\hline\noalign{\smallskip}
$\alpha$&$N$& $\|u-U\|$& rate & $\|u-U\|$ & rate & $\|u-U\|$ & rate & $\|u-U\|$ & rate\\
\noalign{\smallskip}\hline\noalign{\smallskip}
1.1    & 8 & 4.63 1e-4 & -    & 6.15 1e-4 & -    & 3.09 1e-4 & -    & 4.62 1e-5 & -  \\
       & 16& 5.83 1e-5 & 2.99 & 7.72 1e-5 & 2.99 & 3.91 1e-5 & 2.98 & 5.62 1e-6 & 3.04 \\
       & 32& 7.26 1e-6 & 3.00 & 9.63 1e-6 & 3.00 & 4.85 1e-6 & 3.01 & 6.09 1e-7 & 3.21 \\
       & 64& 8.90 1e-7 & 3.03 & 1.09 1e-6 & 3.02 & 5.85 1e-7 & 3.05 & 5.07 1e-8 & 3.59 \\
       &128& 1.06 1e-7 & 3.08 & 1.43 1e-7 & 3.05 & 6.71 1e-8 & 3.13 & 1.88 1e-9 & 4.75 \\
\noalign{\smallskip}\hline\noalign{\smallskip}
1.5    & 8 & 2.21 1e-4 & -    & 2.71 1e-4 & -    & 1.71 1e-4 & -    & 1.47 1e-4 & -  \\
       & 16& 2.82 1e-5 & 2.97 & 3.45 1e-5 & 2.97 & 2.20 1e-5 & 2.96 & 1.88 1e-5 & 2.86 \\
       & 32& 3.54 1e-6 & 3.00 & 4.33 1e-6 & 2.99 & 2.74 1e-6 & 3.00 & 2.33 1e-6 & 3.01 \\
       & 64& 4.33 1e-7 & 3.03 & 5.34 1e-7 & 3.02 & 3.32 1e-7 & 3.04 & 2.81 1e-7 & 3.05 \\
       &128& 5.12 1e-8 & 3.08 & 6.39 1e-8 & 3.06 & 3.85 1e-8 & 3.11 & 3.21 1e-8 & 3.13 \\
\noalign{\smallskip}\hline\noalign{\smallskip}
1.9    & 8 & 3.92 1e-5 & -    & 2.16 1e-5 & -    & 1.02 1e-4 & -    & 5.72 1e-5 & -  \\
       & 16& 4.68 1e-6 & 3.07 & 4.80 1e-6 & 2.17 & 1.43 1e-5 & 2.83 & 7.55 1e-6 & 2.92\\
       & 32& 5.54 1e-7 & 3.08 & 7.08 1e-7 & 2.76 & 1.82 1e-6 & 2.98 & 9.33 1e-7 & 3.02 \\
       & 64& 6.23 1e-8 & 3.15 & 9.80 1e-8 & 2.85 & 2.23 1e-7 & 3.03 & 1.10 1e-8 & 3.08 \\
       &128& 6.12 1e-9 & 3.35 & 1.41 1e-9 & 2.80 & 2.62 1e-8 & 3.09 & 1.21 1e-9 & 3.18 \\
\noalign{\smallskip}\hline
\end{tabular}
\end{table}

\begin{example}\label{example4.3}
There is no challenge introduced if using the idea of this paper and combining with alternating directions methods to solve high dimensional problems. In this example, the following problem
\begin{eqnarray}\label{equation4.3}
\frac{\partial u(x,y,t)}{\partial t}&=&
\,_{0}D_{x}^{\alpha}u(x,y,t)+\,_{x}D_{1}^{\alpha}u(x,y,t)
\nonumber\\
&&+\,_{0}D_{y}^{\beta}u(x,y,t)+\,_{y}D_{1}^{\beta}u(x,y,t)+f(x,y,t)
\end{eqnarray}
is considered in the domain $\Omega=(0,1)^2$ and $t>0$, $1<\alpha,\beta \leq 2$, with the boundary conditions
\begin{equation}
u(x,y,t)=0,~~~(x,y)\in \partial\Omega,~~~t\in[0,1],
\end{equation}
and the initial value
\begin{equation}
u(x,y,0)=x^{3}(1-x)^{3}y^{3}(1-y)^{3},~~~(x,y)\in[0,1]^2.
\end{equation}
The source term is
\begin{eqnarray*}
&&f(x,y,t)
\nonumber\\
&=&-e^{-t}\bigg[x^3(1-x)^3 y^3(1-y)^3
\nonumber\\
&&+\bigg(\frac{\Gamma(4)}{\Gamma(4-\alpha)}\big(x^{3-\alpha}+(1-x)^{3-\alpha}\big)
\nonumber\\
&&-3\times\frac{\Gamma(5)}{\Gamma(5-\alpha)}\big(x^{4-\alpha}+(1-x)^{4-\alpha}\big)
\nonumber\\
&&+3\times\frac{\Gamma(6)}{\Gamma(6-\alpha)}\big(x^{5-\alpha}+(1-x)^{5-\alpha}\big)
\nonumber\\
&&-\frac{\Gamma(7)}{\Gamma(7-\alpha)}\big(x^{6-\alpha}+(1-x)^{6-\alpha}\big)\bigg)y^3(1-y)^3
\nonumber\\
&&+\bigg(\frac{\Gamma(4)}{\Gamma(4-\beta)}\big(y^{3-\beta}+(1-y)^{3-\beta}\big)
\nonumber\\
&&-3\times\frac{\Gamma(5)}{\Gamma(5-\beta)}\big(y^{4-\beta}+(1-y)^{4-\beta}\big)
\nonumber\\
&&+3\times\frac{\Gamma(6)}{\Gamma(6-\beta)}\big(y^{5-\beta}+(1-y)^{5-\beta}\big)
\nonumber\\
&&-\times\frac{\Gamma(7)}{\Gamma(7-\beta)}\big(y^{6-\beta}+(1-y)^{6-\beta}\big)x^4(1-x)^4
\bigg].
\end{eqnarray*}
And the exact solutions is given by $e^{-t}x^{3}(1-x)^3 y^{3}(1-y)^3$.
\end{example}

In Tables \ref{table4.31} and \ref{table4.32}, the quasi-compact Peaceman-Rachford ADI scheme is used to testify the effectiveness of
some second order and third order stable CN-quasi-compact schemes. More ADI schemes can be see in \cite{Tian:12} and \cite{Zhou:13}.
Letting $\tau=h$ and $\tau=h/20$, respectively, in second and third order (in terms of spatial direction) stable schemes
can make sure that the numerical errors caused by the Crank-Nicolson method in time direction is small enough, so that errors in spatial direction are dominant and the convergence rates can be testified.
\begin{table}
\caption{The discrete $L^2$ errors and their convergence rates to Example \ref{example4.3} at $t=1$
by using the second order stable quasi-compact Peaceman-Rachford ADI schemes
for different $\alpha$ and $\beta$ with $\tau=h$.}\label{table4.31}
\begin{tabular}{cccccccccc}
\hline\noalign{\smallskip}
\multicolumn{2}{c}{Number}   &\multicolumn{2}{c}{$4$} &\multicolumn{2}{c}{$5$} &\multicolumn{2}{c}{$8$}
       &\multicolumn{2}{c}{$9$}\\
\noalign{\smallskip}\hline\noalign{\smallskip}
$(\alpha,\beta)$&$N$& $\|u-U\|$& rate & $\|u-U\|$ & rate & $\|u-U\|$ & rate & $\|u-U\|$ & rate\\
\noalign{\smallskip}\hline\noalign{\smallskip}
(1.1,1.9)    & 8 & 1.45 1e-6 & -    & 7.88 1e-7 & -    & 1.54 1e-6 & -    & 1.65 1e-6 & -  \\
             & 16& 3.76 1e-7 & 1.95 & 1.92 1e-7 & 2.04 & 4.26 1e-7 & 1.86 & 4.67 1e-7 & 1.82 \\
             & 32& 9.35 1e-8 & 2.01 & 4.71 1e-8 & 2.03 & 1.13 1e-7 & 1.92 & 1.25 1e-7 & 1.90 \\
             & 64& 2.34 1e-8 & 2.00 & 1.17 1e-8 & 2.01 & 2.97 1e-8 & 1.93 & 3.29 1e-8 & 1.93 \\
             &128& 5.84 1e-9 & 2.00 & 2.93 1e-9 & 2.00 & 7.67 1e-9 & 1.95 & 8.48 1e-9 & 1.96 \\
\noalign{\smallskip}\hline\noalign{\smallskip}
(1.1,1.5)    & 8 & 2.25 1e-6 & -    & 7.82 1e-7 & -    & 2.23 1e-6 & -    & 2.91 1e-6 & -\\
             & 16& 5.58 1e-7 & 2.01 & 1.52 1e-7 & 2.37 & 7.80 1e-7 & 1.64 & 9.34 1e-7 & 1.64\\
             & 32& 1.41 1e-7 & 1.99 & 3.26 1e-8 & 2.22 & 2.22 1e-7 & 1.81 & 2.64 1e-7 & 1.82\\
             & 64& 3.54 1e-8 & 1.99 & 7.54 1e-9 & 2.11 & 6.03 1e-8 & 1.88 & 7.09 1e-8 & 1.90\\
             &128& 8.91 1e-9 & 1.99 & 1.81 1e-9 & 2.06 & 1.58 1e-8 & 1.93 & 1.85 1e-8 & 1.94\\
\noalign{\smallskip}\hline\noalign{\smallskip}
(1.4,1.5)    & 8 & 2.74 1e-6 & -    & 7.75 1e-7 & -    & 2.99 1e-6 & -    & 3.67 1e-6 & -\\
             & 16& 6.82 1e-7 & 2.01 & 1.41 1e-7 & 2.46 & 9.70 1e-7 & 1.62 & 1.15 1e-6 & 1.67 \\
             & 32& 2.72 1e-7 & 1.99 & 2.87 1e-8 & 2.29 & 2.72 1e-7 & 1.84 & 3.18 1e-7 & 1.86\\
             & 64& 4.33 1e-8 & 1.99 & 6.37 1e-9 & 2.17 & 7.19 1e-8 & 1.92 & 8.36 1e-8 & 1.93\\
             &128& 1.01 1e-8 & 1.99 & 1.49 1e-9 & 2.10 & 1.85 1e-8 & 1.96 & 2.14 1e-8 & 1.96\\
\noalign{\smallskip}\hline
\end{tabular}
\end{table}

\begin{table}
\caption{The discrete $L^2$ errors and their convergence rates to Example \ref{example4.3} at $t=1$
by using the third order stable quasi-compact Peaceman-Rachford ADI schemes for different $\alpha$ and $\beta$
with $\tau=h/20$.}\label{table4.32}
\begin{tabular}{cccccccccc}
\hline\noalign{\smallskip}
\multicolumn{2}{c}{Number}   &\multicolumn{2}{c}{$(4,5)$} &\multicolumn{2}{c}{$(5,8)$} &\multicolumn{2}{c}{$(5,9)$}
       &\multicolumn{2}{c}{$(5,10)$}\\
\noalign{\smallskip}\hline\noalign{\smallskip}
$(\alpha,\beta)$&$N$& $\|u-U\|$& rate & $\|u-U\|$ & rate & $\|u-U\|$ & rate & $\|u-U\|$ & rate\\
\noalign{\smallskip}\hline\noalign{\smallskip}
(1.1,1.9)    & 8 & 2.35 1e-7 & -    & 4.59 1e-7 & -    & 4.15 1e-7 & -    & 3.12 1e-7 & -  \\
             & 16& 2.07 1e-8 & 3.50 & 6.64 1e-8 & 2.79 & 5.89 1e-8 & 2.81 & 4.09 1e-8 & 2.93 \\
             & 32& 2.20 1e-9 & 3.24 & 9.24 1e-9 & 2.85 & 8.17 1e-9 & 2.85 & 5.59 1e-9 & 2.87 \\
             & 64& 2.61 1e-10& 3.07 & 1.23 1e-9 & 2.91 & 1.08 1e-9 & 2.91 & 7.41 1e-10& 2.92 \\
             &128& 3.27 1e-11& 3.00 & 1.59 1e-10& 2.95 & 1.40 1e-10& 2.95 & 9.55 1e-11& 2.96 \\
\noalign{\smallskip}\hline\noalign{\smallskip}
(1.1,1.5)    & 8 & 3.20 1e-7 & -    & 9.31 1e-7 & -    & 8.56 1e-7 & -    & 6.60 1e-7 & -  \\
             & 16& 4.41 1e-8 & 2.86 & 1.37 1e-7 & 2.76 & 1.26 1e-7 & 2.77 & 9.55 1e-8 & 2.79 \\
             & 32& 6.04 1e-9 & 2.87 & 1.87 1e-8 & 2.88 & 1.71 1e-8 & 2.88 & 1.29 1e-8 & 2.88 \\
             & 64& 8.20 1e-10& 2.88 & 2.46 1e-9 & 2.93 & 2.25 1e-9 & 2.93 & 1.71 1e-9 & 2.92 \\
             &128& 1.10 1e-10& 2.90 & 3.18 1e-10& 2.95 & 2.91 1e-10& 2.95 & 2.23 1e-10& 2.94 \\
\noalign{\smallskip}\hline\noalign{\smallskip}
(1.4,1.5)    & 8 & 3.86 1e-7 & -    & 6.89 1e-7 & -    & 6.50 1e-7 & -    & 5.63 1e-7 & -  \\
             & 16& 5.16 1e-8 & 2.90 & 9.74 1e-8 & 2.82 & 9.15 1e-8 & 2.83 & 7.84 1e-8 & 2.84\\
             & 32& 6.92 1e-9 & 2.90 & 1.32 1e-8 & 2.89 & 1.24 1e-8 & 2.89 & 1.06 1e-8 & 2.89 \\
             & 64& 9.27 1e-10& 2.90 & 1.75 1e-9 & 2.91 & 1.65 1e-9 & 2.91 & 1.41 1e-9 & 2.91 \\
             &128& 1.23 1e-10& 2.91 & 2.31 1e-10& 2.92 & 2.17 1e-10& 2.92 & 1.87 1e-10& 2.92 \\
\noalign{\smallskip}\hline
\end{tabular}
\end{table}

\section{Conclusion}\label{sec:5}

This paper focuses on developing a series quasi-compact schemes for space fractional diffusion equations; and the schemes can be easily extended to general space fractional PDEs. The so-called quasi-compactness for fractional PDEs which are nonlocal means that while the high order schemes are designed, no points outside of the domain of the solution are used. The ideas of developing the schemes are based on the superconvergence of the Gr\"{u}nwald approximation to Riemann-Liouville derivative at a special point. And the strategy to design any desired high order scheme is presented. The detailed stability and convergence analysis are performed for some of the derived schemes. The extensive numerical experiments, including one and two dimensional equations, are performed to show the effectiveness of the derived schemes. In particular, the techniques for treating the nonhomogeneous boundary conditions are introduced, which can keep the potential high accuracy of the schemes and at the same time do not need to specify any more non-physical conditions on the boundaries.


\begin{acknowledgements}
The authors thank WenYi Tian for the discussions.
\end{acknowledgements}

\appendix
\renewcommand{\appendixname}{Appendix~\Alph{section}}
\section{Appendix}

Here we show in detail how to apply the approximations discussed in this paper to the Riemann-Liouville derivatives of a function with nonhomogeneous boundaries. And then we derive a general numerical scheme to the nonhomogeneous steady state problem
\begin{equation}\label{equation_add1.1}
\left\{ \begin{array}{l}
-\,_{0}D_{x}^{\alpha}u(x)+b(x)u(x)=f(x),~~~x \in (0,1),\\
u(0)=\phi_{0},\\
u(1)=\phi_{1}
\end{array} \right.
\end{equation}
with $1<\alpha<2$, $b(x)\geq 0$. Combining the obtained scheme and the discretization of time derivative leads to the scheme for the time-dependent space fractional PDE with nonhomogeneous boundary conditions; see numerical results of Example \ref{example4.add}: Tables \ref{table_add1.1} and \ref{table_add1.2}.

\subsection{Derivation of the general numerical scheme}\label{subsec_add1.1}

To begin, a fundamental lemma is listed as follows.

\begin{lemma}\label{lemma_add1.1}
If $y\in C^n[a,b],~D^{(n+1)}y \in L^1[a,b]$, then for any $h>0$, $x_j=jh+a$ and some given positive integer $n$, there exists an unique set of $\{a_j^l\}_{j=0}^{n+l}$ such that
\begin{equation}
\sum_{j=0}^{n+l} a_j^l y(x_j)=\frac{y^{(l)}(x_0)}{l!}h^l+O(h^{n+l+1})
\end{equation}
i.e.
\begin{equation}\label{equation_add1.2}
\sum_{j=0}^{n+l} h^{-l} a_j^l y(x_j)=\frac{y^{(l)}(x_0)}{l!}+O(h^{n+1})
\end{equation}
holds, for $l=0,1,\cdots,n$.
\end{lemma}
\begin{proof}
Under the hypotheses of Lemma \ref{lemma_add1.1}, the following equalities can be obtained from the Taylor expansion
\begin{equation}\label{equation_add1.3}
y(x_j)=\sum_{l=0}^n \frac{y^{(l)}(x_0)}{l!}(jh)^l+O(h^{n+1}),~~~j=0,1,\cdots,n+l.
\end{equation}
Denote $\textbf{a}^l=(a_0^l,a_1^l,\cdots,a_{n+l}^l)^{T}$. It is easy to see that $\textbf{a}^l$ is the solution of the equations
\begin{equation}\label{equation_add1.4}
\textbf{L}~\textbf{a}^l=\textbf{e}^l,
\end{equation}
where $\textbf{e}^l$ is a $(n+l+1)$-dimensional vector, with
\begin{equation}\label{equation_add1.5}
\textbf{e}^l_j=\left\{ \begin{array}{ll}
1,&~j=l,\\
0,&~else;
\end{array} \right.
\end{equation}
and
\begin{equation}\label{equation_add1.6}
\textbf{L}=\left[ \begin{array}{ccccc}
1      & 1      & 1      & \cdots & 1      \\
0      & 1      & 2      & \cdots & n+l      \\
0      & 1      & 2^2    & \cdots & (n+l)^2    \\
\vdots & \vdots & \vdots &        & \vdots \\
0      & 1      & 2^{n+l}    & \cdots & (n+l)^{n+l}
\end{array}
\right].
\end{equation}
The existence and uniqueness of $\textbf{a}^l$ are guaranteed by the reversibility of \textbf{L}, which complete the proof.
\end{proof}

\begin{remark}
While it is true that the condition number of $\textbf{L}$ in Lemma \ref{lemma_add1.1} grows rapidly with the increase of $n$,
which might bring an inaccurate numerical result of $\{a_j^l\}_{j=0}^{n+l}$. Fortunately, however, in practice, $n$ is related to the convergence order of the approximations, which usually is less than $5$. Therefore, most of the time we do not need to take the extreme case of $\textbf{L}$ into consideration.
\end{remark}

For the easy of presentation, it is convenient to assume $a_j^l=0$ for $j=n+l+1,\cdots,2n$. In this way, (\ref{equation_add1.2}) can be
uniformly rewritten as
\begin{equation}\label{equation_add1.2}
\sum_{j=0}^{2n} h^{-l} a_j^l y(x_j)=\frac{y^{(l)}(x_0)}{l!}+O(h^{n+1}).
\end{equation}

Let us assume $u\in C^n[0,b]$, $D^{n+1}u(x)\in L^1[0,b]$, but not necessarily with homogeneous boundaries.
It is obvious that
\begin{eqnarray}\label{equation_add1.7}
\,_{0}D_{x}^{\alpha}u(x)&=&
\,_{0}D_{x}^{\alpha}\big(u(x)-\sum_{l=0}^n \frac{u^{(l)}(0)}{l!}x^l\big)+
\sum_{l=0}^n \frac{u^{(l)}(0)}{l!}\,_{0}D_{x}^{\alpha}x^l
\nonumber\\
&:=&\,_{0}D_{x}^{\alpha} r(x)+\sum_{l=0}^n \frac{u^{(l)}(0)}{l!}\,_{0}D_{x}^{\alpha} x^l,
\end{eqnarray}
where $r(x)=u(x)-\sum_{l=0}^n \frac{u^{(l)}(0)}{l!}x^l$ satisfies $r(x)\in C^n [0,b]$, $D^{n+1}r(x)\in L^1[0,b]$, also $D^l r(0)=0$ for $l=0,1,\cdots,n$.

Thus, by Theorem \ref{the2.23}, there are two sets of data $\{c_{-1}, c_0, c1\}$ and $\{d_{-1},d_0,d_1\}$, such that
\begin{eqnarray}\label{equation_add1.8}
&&c_{-1}\,_{0}D_{x}^{\alpha}r(x-h)+c_0\,_{0}D_{x}^{\alpha}r(x)+c_1\,_{0}D_{x}^{\alpha}r(x+h)
\nonumber\\
&=&
d_{-1}\delta_{h,-1}^{\alpha}r(x)+d_{0}\delta_{h,0}^{\alpha}r(x)+d_{1}\delta_{h,1}^{\alpha}r(x)+O(h^{n-1}).
\end{eqnarray}

Taking
\begin{equation}\label{equation_add1.9}
ds(x_i):=h^{-\alpha}\sum_{j=0}^{2n}\sum_{l=0}^n \frac{\Gamma(l+1)}{\Gamma(l+1-\alpha)}i^{l-\alpha}a_j^lu(x_j),
\end{equation}
then by Lemma \ref{lemma_add1.1} and (\ref{equation2.13}), we have
\begin{equation}\label{equation_add1.10}
ds(x_i)-\sum_{l=0}^n \frac{u^{(l)}(0)}{l!}\,_{0}D_{x}^{\alpha}x_{i}^l=O(h^{n+1}),
\end{equation}
and
\begin{eqnarray}\label{equation_add1.11}
&&\delta_{h,p}^{\alpha}r(x_i)
\nonumber\\
&=&h^{-\alpha}\sum_{k=0}^{i+p}g_k^{(\alpha)} r(x_{i-k+p})
\nonumber\\
&=&h^{-\alpha}\Big[\sum_{k=0}^{i+p}g_k^{(\alpha)} u(x_{i-k+p})-\sum_{k=0}^{i+p}g_k^{(\alpha)}\sum_{l=0}^n\frac{u^{(l)}(0)}{l!}x_{i-k+p}^{l}\Big]
\nonumber\\
&=&h^{-\alpha}\Big[\sum_{k=0}^{i+p}g_k^{(\alpha)} u(x_{i-k+p})-
\sum_{k=0}^{i+p}g_k^{(\alpha)}\sum_{l=0}^n\sum_{j=0}^{2n} a_j^l u(x_j)(i-k+p)^l+\sum_{k=0}^{i+p}g_k^{(\alpha)}\sum_{l=0}^n O(h^{n+1})\Big]
\nonumber\\
&=&h^{-\alpha}\Big[\sum_{k=0}^{i+p}g_k^{(\alpha)} u(x_{i-k+p})-
\sum_{k=0}^{i+p}g_k^{(\alpha)}\sum_{l=0}^n\sum_{j=0}^{2n} a_j^l u(x_j)(i-k+p)^l+O(h^{n+1})\Big]
\nonumber\\
&=&h^{-\alpha}\Big[\sum_{k=0}^{i+p}g_k^{(\alpha)} u(x_{i-k+p})-\sum_{j=0}^{2n}\big(\sum_{l=0}^n  (\sum_{k=0}^{i+p}g_k^{(\alpha)}(i-k+p)^{l})\big)a_j^l u(x_j)\Big]
\nonumber\\
&&+O(h^{n+1-\alpha})
\end{eqnarray}
for $1\leq i \leq N=1/h$. Hence,
\begin{eqnarray}\label{equation_add1.12}
&&c_{-1}\,_{0}D_{x}^{\alpha}u(x_{i-1})+c_0\,_{0}D_{x}^{\alpha}u(x_{i})+c_1\,_{0}D_{x}^{\alpha}u(x_{i+1})
\nonumber\\
&=&
d_{-1}\delta_{h,-1}^{\alpha}r(x_{i})+d_{0}\delta_{h,0}^{\alpha}r(x_i)+d_{1}\delta_{h,1}^{\alpha}r(x_i)
\nonumber\\
&&+c_{-1}ds(x_{i-1})+c_{0}ds(x_{i})+c_{1}ds(x_{i+1})+O(h^{n-1})
\nonumber\\
&=&d_{-1}h^{-\alpha}\Big[\sum_{k=0}^{i-1}g_k^{(\alpha)} u(x_{i-k-1})-\sum_{j=0}^{2n}\big(\sum_{l=0}^n  (\sum_{k=0}^{i-1}g_k^{(\alpha)}(i-k-1)^{l})\big)a_j^l u(x_j)\Big]
\nonumber\\
&&+d_{0}h^{-\alpha}\Big[\sum_{k=0}^{i}g_k^{(\alpha)} u(x_{i-k})-\sum_{j=0}^{2n}\big(\sum_{l=0}^n  (\sum_{k=0}^{i}g_k^{(\alpha)}(i-k)^{l})\big)a_j^l u(x_j)\Big]
\nonumber\\
&&+d_{1}h^{-\alpha}\Big[\sum_{k=0}^{i+1}g_k^{(\alpha)} u(x_{i-k+1})-\sum_{j=0}^{2n}\big(\sum_{l=0}^n  (\sum_{k=0}^{i+1}g_k^{(\alpha)}(i-k+1)^{l})\big)a_j^l u(x_j)\Big]
\nonumber\\
&&+c_{-1}h^{-\alpha}\sum_{j=0}^{2n}\sum_{l=0}^n \frac{\Gamma(l+1)}{\Gamma(l+1-\alpha)}(i-1)^{l-\alpha}a_j^l u(x_j)
\nonumber\\
&&+c_{0}h^{-\alpha}\sum_{j=0}^{2n}\sum_{l=0}^n \frac{\Gamma(l+1)}{\Gamma(l+1-\alpha)}i^{l-\alpha}a_j^l u(x_j)
\nonumber\\
&&+c_{1}h^{-\alpha}\sum_{j=0}^{2n}\sum_{l=0}^n \frac{\Gamma(l+1)}{\Gamma(l+1-\alpha)}(i+1)^{l-\alpha}a_j^l u(x_j)+O(h^{n-1}).
\end{eqnarray}

So far, by denoting $u_{i}$ as the numerical approximation of $u(x_i)$, the corresponding
scheme for (\ref{equation_add1.1}) is obtained as
\begin{eqnarray}\label{equation_add1.13}
&&-d_{-1}h^{-\alpha}\Big[\sum_{k=0}^{i-1}g_k^{(\alpha)} u_{i-k-1}-\sum_{j=0}^{2n}\big(\sum_{l=0}^n  (\sum_{k=0}^{i-1}g_k^{(\alpha)}(i-k-1)^{l})\big)a_j^l u_j\Big]
\nonumber\\
&&-d_{0}h^{-\alpha}\Big[\sum_{k=0}^{i}g_k^{(\alpha)} u_{i-k}-\sum_{j=0}^{2n}\big(\sum_{l=0}^n  (\sum_{k=0}^{i}g_k^{(\alpha)}(i-k)^{l})\big)a_j^l u_j\Big]
\nonumber\\
&&-d_{1}h^{-\alpha}\Big[\sum_{k=0}^{i+1}g_k^{(\alpha)} u_{i-k+1}-\sum_{j=0}^{2n}\big(\sum_{l=0}^n  (\sum_{k=0}^{i+1}g_k^{(\alpha)}(i-k+1)^{l})\big)a_j^l u_j\Big]
\nonumber\\
&&-c_{-1}h^{-\alpha}\sum_{j=0}^{2n}\sum_{l=0}^n \frac{\Gamma(l+1)}{\Gamma(l+1-\alpha)}(i-1)^{l-\alpha}a_j^l u_j
\nonumber\\
&&-c_{0}h^{-\alpha}\sum_{j=0}^{2n}\sum_{l=0}^n \frac{\Gamma(l+1)}{\Gamma(l+1-\alpha)}i^{l-\alpha} a_j^l u_j
\nonumber\\
&&-c_{1}h^{-\alpha}\sum_{j=0}^{2n}\sum_{l=0}^n \frac{\Gamma(l+1)}{\Gamma(l+1-\alpha)}(i+1)^{l-\alpha} a_j^l u_j
\nonumber\\
&&+c_{-1}b_{i-1}u_{i-1}+c_0 b_{i}u_{i}+c_1 b_{i+1}u_{i+1}
\nonumber\\
&=&c_{-1}f_{i-1}+c_0 f_{i}+c_{1}f_{i+1}
\end{eqnarray}
for $i=1,2,\cdots,N-1$, where $b_{i}=b(x_i),~f_i=f(x_i)$ for $i=0,1,\cdots,N$.
And (\ref{equation_add1.13}) can be rewritten in matrix form as
\begin{eqnarray}\label{equation_add1.14}
&&-d_{-1}h^{-\alpha}\Big[\big(\textbf{G}_{-1}\textbf{u}+u_0\textbf{g}_{-1} \big)
-\big(\textbf{S}_{-1}\textbf{u}+u_0\textbf{r}_{s,-1} \big)\Big]
\nonumber\\
&&-d_{0}h^{-\alpha}\Big[\big(\textbf{G}_{0}\textbf{u}+u_0\textbf{g}_{0} \big)
-\big(\textbf{S}_{0}\textbf{u}+u_0\textbf{r}_{s,0} \big)\Big]
\nonumber\\
&&-d_{1}h^{-\alpha}\Big[\big(\textbf{G}_{1}\textbf{u}+u_0\textbf{g}_{1}+u_{N}\textbf{e}_{N-1}\big)
-\big(\textbf{S}_{1}\textbf{u}+u_0\textbf{r}_{s,1} \big)\Big]
\nonumber\\
&&-c_{-1}h^{-\alpha}\Big[\textbf{D}_{-1}\textbf{u}+u_0\textbf{r}_{d,-1}\Big]
-c_{0}h^{-\alpha}\Big[\textbf{D}_{0}\textbf{u}+u_0\textbf{r}_{d,0}\Big]
\nonumber\\
&&-c_{1}h^{-\alpha}\Big[\textbf{D}_{1}\textbf{u}+u_0\textbf{r}_{d,1}\Big]
+\textbf{B}\textbf{u}+\textbf{r}_b
\nonumber\\
&=&c_{-1}\textbf{f}_{-1}+c_{0}\textbf{f}_{0}+c_{1}\textbf{f}_{1},
\end{eqnarray}
where
\begin{eqnarray*}\label{equation_add1.15}
\textbf{G}_{-1}=\left [ \begin{array}{cccccc}
0&&&&&\\
g_{0}^{(\alpha)}   &  0 &       &&&            \\
g_{1}^{(\alpha)}  &   g_{0}^{(\alpha)}  & 0  &&&         \\
        &\ddots  &  \ddots    &   \ddots && \\
g_{N-4}^{(\alpha)} & g_{N-5}^{(\alpha)}& \cdots  &  g_{0}^{(\alpha)} & 0& \\
g_{N-3}^{(\alpha)} & g_{N-4}^{(\alpha)} & \cdots  & g_{1}^{(\alpha)}& g_{0}^{(\alpha)}  &  0
\end{array}
\right ],
\end{eqnarray*}

\begin{eqnarray*}\label{equation_add1.16}
\textbf{G}_{0}=\left [ \begin{array}{cccccc}
g_{0}^{(\alpha)}   &   &       &&&            \\
g_{1}^{(\alpha)}  &   g_{0}^{(\alpha)}  &   &&&         \\
        &\ddots  &  \ddots    &   \ddots && \\
g_{N-3}^{(\alpha)} & g_{N-4}^{(\alpha)} & \cdots  & g_{0}^{(\alpha)}& & \\
g_{N-2}^{(\alpha)} & g_{N-3}^{(\alpha)} & \cdots  & g_{1}^{(\alpha)}& g_{0}^{(\alpha)}  &
\end{array}
\right ],
\end{eqnarray*}

\begin{eqnarray*}\label{equation_add1.17}
\textbf{G}_{1}=\left [ \begin{array}{cccccc}
g_{1}^{(\alpha)}   &  g_{0}^{(\alpha)}  &       &&&            \\
g_{2}^{(\alpha)}  &   g_{1}^{(\alpha)}  &g_{0}^{(\alpha)}   &&&         \\
        &\ddots  &  \ddots    &   \ddots && \\
g_{N-2}^{(\alpha)} & g_{N-3}^{(\alpha)}& \cdots  &  g_{1}^{(\alpha)} & g_{0}^{(\alpha)} \\
g_{N-1}^{(\alpha)} & g_{N-2}^{(\alpha)} & \cdots  & g_{2}^{(\alpha)}& g_{1}^{(\alpha)}
\end{array}
\right ];
\end{eqnarray*}

\begin{eqnarray*}\label{equation_add1.18}
&&\textbf{u}=(u_1,u_2,\cdots,u_{N-1})^{T};\\
&&\textbf{g}=(g_0,g_1,\cdots,g_N)^{T},\\
&&\textbf{g}_{-1}=\textbf{g}(1:N-1),
~\textbf{g}_{0}=\textbf{g}(2:N),
~\textbf{g}_{1}=\textbf{g}(3:N+1);
\end{eqnarray*}
$\textbf{e}_{N-1}$ is a $(N-1)$-dimensional vector, with
\begin{eqnarray*}\label{equation_add1.19}
(\textbf{e}_{N-1})_j=\left\{ \begin{array}{ll}
1,&~j=N-1,\\
0,&~else;
\end{array} \right.
\end{eqnarray*}

\begin{eqnarray*}\label{equation_add1.20}
\textbf{P}=\left [ \begin{array}{ccccc}
1      &     1  & 1      & \cdots  & 1\\
1      &     2  &     2^2 & \cdots & 2^n\\
\vdots & \vdots &  \vdots &        & \vdots\\
1      & N-1    & (N-1)^2 & \cdots & (N-1)^n
\end{array}
\right ],
\end{eqnarray*}

\begin{eqnarray*}\label{equation_add1.21}
\textbf{Q}_{-1}=[\textbf{g}_{-1},\textbf{0},\cdots,\textbf{0}],
~\textbf{Q}_{0}=[\textbf{g}_{0},\textbf{0},\cdots,\textbf{0}],
\end{eqnarray*}

\begin{eqnarray*}\label{equation_add1.22}
\textbf{Q}_{1}=\left [ \begin{array}{ccccc}
g_{2}^{(\alpha)}  &   &     &        &\\
g_{3}^{(\alpha)}  &   &     &        & \\
\vdots            &   &     &        &\\
g_{N}^{(\alpha)}+1& N & N^2 & \cdots & N^n
\end{array}
\right ],
\end{eqnarray*}

\begin{eqnarray*}\label{equation_add1.23}
\textbf{A}_{(n+1)\times N}=\left [ \begin{array}{cccccccc}
a_0^0 & a_1^0 & \cdots & a_{2n}^0 & 0      & \cdots & 0 &\\
a_0^1 & a_1^1 & \cdots & a_{2n}^1 & 0      & \cdots & 0 &\\
\vdots&\vdots &        &\vdots & \vdots &        & \vdots\\
a_0^n & a_1^n & \cdots & a_{2n}^n & 0      & \cdots & 0 &\\
\end{array}
\right ];
\end{eqnarray*}

\begin{eqnarray*}\label{equation_add1.24}
&&\textbf{M}_{id}=(\textbf{G}_{id}\textbf{P}+\textbf{Q}_{id})\textbf{A},\\
&&\textbf{S}_{id}=\textbf{M}_{id}(:,2:end),
~\textbf{r}_{s,id}=\textbf{M}_{id}(:,1), ~~~id=-1,~0,~1;
\end{eqnarray*}

\begin{eqnarray*}\label{equation_add1.25}
&&(\textbf{K}_{id})_{i,l}=\frac{\Gamma(l+1)}{\Gamma(l+1-\alpha)}(i+id)^{l-\alpha},~i=1,2,\cdots,N-1,~l=0,1,\cdots,n,\\
&&\textbf{N}_{id}=\textbf{K}_{id}\cdot \textbf{A},\\
&&\textbf{D}_{id}=\textbf{N}_{id}(:,2:end),
~\textbf{r}_{d,id}=\textbf{N}_{id}(:,1), ~~~id=-1,~0,~1;
\end{eqnarray*}

\begin{eqnarray*}\label{equation_add1.26}
\textbf{B}=\left[ \begin{array}{ccccc}
c_{0}b_1   &   c_{1}b_2  &                 &&           \\
c_{-1}b_1  &   c_{0}b_2  &  c_{1}b_3       &&          \\
           &   c_{-1}b_2 &  c_{0}b_3       &  c_{1}b_4        &  \\
           &\ddots       &  \ddots         &   \ddots         &\\
           &             &  c_{-1}b_{N-3}  &   c_{0}b_{N-2}   &  c_{1}b_{N-1} \\
           &             &                 &   c_{-1}b_{N-2}  &  c_{0}b_{N-1}
\end{array}
\right];
\end{eqnarray*}

\begin{eqnarray*}\label{equation_add1.27}
&&\textbf{r}_b=(c_1 b_0 u_0,0,\cdots,0,c_1 b_N u_N)^{T};\\
&&\textbf{F}=(f_0,f_1,\cdots,f_N)^{T},\\
&&\textbf{F}_{-1}=\textbf{F}(1:N-1),
~\textbf{F}_{0}=\textbf{F}(2:N),
~\textbf{F}_{1}=\textbf{F}(3:N+1).
\end{eqnarray*}

Take
\begin{equation}\label{equation_add1.28}
ds_j:=h^{-\alpha}\sum_{j=0}^n\sum_{l=0}^n \frac{\Gamma(l+1)}{\Gamma(l+1-\alpha)}i^{l-\alpha}a_j^l u_j.
\end{equation}
It should be declared if $u(0)\neq 0$ or $u'(0)\neq 0$, then $\,_{0}D_{x}^{\alpha}u(x)\rightarrow \infty(x\rightarrow 0)$,
$f(x)\rightarrow\infty (x\rightarrow 0)$; else if $u(0)=0$ and $u'(0)=0$, then $\,_{0}D_{x}^{\alpha}u(x)\rightarrow 0 (x\rightarrow 0)$, $f(x)\rightarrow 0 (x\rightarrow 0)$. Considering these, $ds_0$ is mandatorily to be $0$, and $f_0$ to be $b_0 u_0$, in real computation.

%

\subsection{Examples}\label{subsec_add1.2}

We apply the quasi-compact approximations to the steady state problem
\begin{equation}\label{equation_add1.29}
\left\{ \begin{array}{l}
-\,_{0}D_{x}^{\alpha}u(x)+u(x)=f(x),~~~x \in (0,1),\\
u(0)=-1,\\
u(1)=-3,
\end{array} \right.
\end{equation}
with the right-hand function
\begin{equation}\label{equation_add1.30}
f(x)=\frac{1}{\Gamma(1-\alpha)}x^{-\alpha}+\frac{1}{\Gamma(2-\alpha)}x^{1-\alpha}+\frac{\Gamma(4+\alpha)}{\Gamma(4)}x^3-(1+x+x^{3+\alpha}).
\end{equation}
And the exact solution is $u(x)=-1-x-x^{3+\alpha}$.
The numerical data in Tables \ref{table_add1.1} and \ref{table_add1.2} verify the desired convergence. 

\begin{table}
\caption{The discrete $L^2$ errors and their convergence rates to (\ref{equation_add1.29})
by using the second and fourth order quasi-compact approximations
for different $\alpha$}\label{table_add1.1}
\begin{tabular}{cccccccccc}
\hline\noalign{\smallskip}
\multicolumn{2}{c}{Number}   &\multicolumn{2}{c}{$1$} &\multicolumn{2}{c}{$3$}
       &\multicolumn{2}{c}{$5$} &\multicolumn{2}{c}{$(1,2)+(1,8)$} \\
\noalign{\smallskip}\hline\noalign{\smallskip}
$\alpha$&$N$& $\|u-U\|$& rate & $\|u-U\|$ & rate & $\|u-U\|$ & rate & $\|u-U\|$ & rate\\
\noalign{\smallskip}\hline\noalign{\smallskip}
1.1    & 8 & 1.96 1e-3 & -    & 1.11 1e-2 & -    & 3.90 1e-3 & -    & 6.12 1e-6 & -  \\
       & 16& 4.85 1e-4 & 2.01 & 3.10 1e-3 & 1.85 & 8.75 1e-4 & 2.16 & 2.67 1e-7 & 5.52 \\
       & 32& 1.17 1e-4 & 2.05 & 8.16 1e-4 & 1.92 & 2.02 1e-4 & 2.11 & 1.34 1e-8 & 4.55 \\
       & 64& 2.82 1e-5 & 2.05 & 2.09 1e-4 & 1.96 & 4.80 1e-5 & 2.07 & 4.74 1e-10& 4.59 \\
\noalign{\smallskip}\hline\noalign{\smallskip}
1.5    & 8 & 3.96 1e-3 & -    & 1.23 1e-2 & -    & 3.06 1e-3 & -    & 1.45 1e-5 & -\\
       & 16& 1.04 1e-3 & 1.93 & 3.13 1e-3 & 1.97 & 7.88 1e-4 & 1.96 & 5.83 1e-7 & 4.63\\
       & 32& 2.61 1e-4 & 1.99 & 7.82 1e-4 & 2.00 & 2.16 1e-4 & 1.87 & 2.18 1e-8 & 4.74\\
       & 64& 6.43 1e-5 & 2.02 & 1.94 1e-4 & 2.01 & 6.83 1e-5 & 1.66 & 7.79 1e-10& 4.81\\
\noalign{\smallskip}\hline\noalign{\smallskip}
1.9    & 8 & 6.04 1e-3 & -    & 8.01 1e-3 & -    & 4.21 1e-3 & -    & 6.92 1e-6 & -\\
       & 16& 1.53 1e-3 & 1.98 & 2.04 1e-3 & 1.98 & 1.07 1e-3 & 1.97 & 2.59 1e-7 & 4.74 \\
       & 32& 3.87 1e-4 & 1.99 & 5.15 1e-4 & 1.99 & 2.71 1e-4 & 1.99 & 8.96 1e-9 & 4.85\\
       & 64& 9.71 1e-5 & 1.99 & 1.29 1e-4 & 1.99 & 6.80 1e-5 & 1.99 & 3.04 1e-10& 4.88\\
\noalign{\smallskip}\hline
\end{tabular}
\end{table}

\begin{table}
\caption{The discrete $L^2$ errors and their convergence rates to (\ref{equation_add1.29})
by using the third order quasi-compact approximations
for different $\alpha$}\label{table_add1.2}
\begin{tabular}{cccccccccc}
\hline\noalign{\smallskip}
\multicolumn{2}{c}{Number}   &\multicolumn{2}{c}{$$(1,3)$$} &\multicolumn{2}{c}{$$(1,5)$$} &\multicolumn{2}{c}{$$(3,5)$$}
       &\multicolumn{2}{c}{$$(3,8)$$}\\
\noalign{\smallskip}\hline\noalign{\smallskip}
$\alpha$&$N$& $\|u-U\|$& rate & $\|u-U\|$ & rate & $\|u-U\|$ & rate & $\|u-U\|$ & rate\\
\noalign{\smallskip}\hline\noalign{\smallskip}
1.1    & 8 & 6.84 1e-4 & -    & 2.59 1e-4 & -    & 4.73 1e-4 & -    & 4.33 1e-4 & -  \\
       & 16& 8.68 1e-5 & 2.98 & 5.84 1e-5 & 2.98 & 6.10 1e-5 & 2.96 & 5.27 1e-5 & 3.04 \\
       & 32& 1.09 1e-5 & 2.99 & 7.37 1e-6 & 2.99 & 7.74 1e-6 & 2.98 & 6.56 1e-6 & 3.01 \\
       & 64& 1.38 1e-6 & 2.99 & 2.97 1e-7 & 2.99 & 9.75 1e-7 & 2.99 & 8.24 1e-7 & 2.99 \\
\noalign{\smallskip}\hline\noalign{\smallskip}
1.5    & 8 & 4.07 1e-4 & -    & 3.15 1e-4 & -    & 1.00 1e-5 & -    & 4.31 1e-4 & -\\
       & 16& 5.20 1e-5 & 2.97 & 4.05 1e-5 & 2.96 & 2.08 1e-6 & 2.27 & 5.54 1e-5 & 2.96\\
       & 32& 6.58 1e-6 & 2.98 & 5.13 1e-6 & 2.98 & 3.01 1e-7 & 2.79 & 7.02 1e-6 & 2.98\\
       & 64& 8.28 1e-7 & 2.99 & 6.45 1e-7 & 2.99 & 3.92 1e-8 & 2.94 & 8.83 1e-7 & 2.99\\
\noalign{\smallskip}\hline\noalign{\smallskip}
1.9    & 8 & 8.60 1e-5 & -    & 2.21 1e-4 & -    & 6.98 1e-4 & -    & 2.64 1e-4 & -\\
       & 16& 1.05 1e-5 & 3.03 & 3.12 1e-5 & 2.82 & 9.11 1e-5 & 2.94 & 2.67 1e-5 & 3.07 \\
       & 32& 1.30 1e-6 & 3.01 & 4.02 1e-6 & 2.96 & 1.15 1e-5 & 2.98 & 3.01 1e-6 & 3.01\\
       & 64& 1.63 1e-7 & 3.00 & 5.06 1e-7 & 2.99 & 1.45 1e-6 & 3.00 & 4.12 1e-7 & 3.00\\
\noalign{\smallskip}\hline
\end{tabular}
\end{table}




\end{document}